\documentclass{article}
\usepackage{bbm}
\usepackage{caption}
\usepackage{rotating}
\usepackage[maxfloats=100]{morefloats}
\usepackage{listings}
\usepackage{booktabs}
\usepackage{xcolor}
\usepackage[title]{appendix}
\usepackage{makecell} 
\usepackage{multirow}

\usepackage[symbol]{footmisc}

\usepackage{longtable}
\usepackage{subcaption}
\usepackage[numbers]{natbib}
\usepackage{geometry}

\usepackage{verbatim}
\usepackage{xcolor}
\usepackage{graphicx}
\usepackage{enumitem}
\usepackage{comment}
\usepackage{subfiles}

\usepackage{todonotes}
\usepackage{tikz-qtree,tikz-qtree-compat}
\usepackage{scrextend}
\usepackage{framed}
\usepackage{placeins}
\usepackage{setspace}
 \usepackage[ruled,vlined]{algorithm2e}
\usepackage{mathtools,collcell,eqparbox}

\usepackage{multicol}
\usepackage{pgfplots}
\usepackage{tikz}
\usepackage{listings}
\usepackage{courier}
\usepackage{url}
\usepackage{color}
\usepackage[title]{appendix}
\usepackage{authblk}

\usepackage{amssymb}
\usepackage{amsthm}
\usepackage{ dsfont }
\usepackage{amsmath}

\usepackage[pagebackref = true]{hyperref}
\hypersetup{
  colorlinks = true,
  linkcolor = teal,
  anchorcolor = teal,
  citecolor = teal,
  filecolor = teal,
  urlcolor = teal
  }

\pgfplotsset{compat=1.15}
\newcounter{BMatrix}

\newcommand{\setmaxwd}[1]{%
 \eqmakebox[BM-\theBMatrix][\BMalign]{$#1$}%
}
\MHInternalSyntaxOn

\MHInternalSyntaxOff
\makeatother

\newtheorem{theorem}{Theorem}
\newtheorem{assumption}{Assumption}

\newtheorem{proposition}{Proposition}[section]

\newtheorem{corollary}{Corollary}[theorem]
\newtheorem{lemma}{Lemma}

\theoremstyle{definition}
\newtheorem{definition}{Definition}
\newtheorem{example}{Example}

\theoremstyle{remark}
\newtheorem{remark}{Remark}
\newtheorem*{remark*}{Remark}

\newcommand{\inv}{^{-1}}

\newcommand{\1}{\mathds{1}}

\newcommand{\R}{\mathbb{R}}

\newcommand{\Z}{\mathbb{Z}}
\newcommand{\N}{\mathbb{N}}

\newcommand{\ra}{\rightarrow}

\newcommand{\cd}{\cdot}
\newcommand{\ds}{\dots}
\newcommand{\card}{\mathrm{card}}

\newcommand{\mrm}[1]{\mathrm{#1}}

\newcommand{\PiH}{\Pi_{\mathrm{H}}}

\newcommand{\PiS}{\Pi_{\mathrm{S}}}

\newcommand{\piRL}{{\pi_{\mathrm{RL}}}}
\newcommand{\PiRL}{{\Pi_{\mathrm{RL}}}}

\newcommand{\cA}{\mathcal{A}}
\newcommand{\cB}{\mathcal{B}}
\newcommand{\cC}{\mathcal{C}}

\newcommand{\cF}{\mathcal{F}}

\newcommand{\cH}{\mathcal{H}}

\newcommand{\cP}{\mathcal{P}}
\newcommand{\cQ}{\mathcal{Q}}

\newcommand{\cS}{\mathcal{S}}

\newcommand{\spnorm}[1]{\left|#1\right|_\mathrm{span}}
\newcommand{\set}[1]{\left\{{#1}\right\}}

\newcommand{\floor}[1]{\left\lfloor{#1}\right\rfloor}

\newcommand{\abs}[1]{\left|#1\right|}
\newcommand{\sqbk}[1]{\left[ #1 \right]}
\newcommand{\sqbkcond}[2]{\left[ #1 \middle| #2 \right]}

\newcommand{\crbk}[1]{\left( #1 \right)}

\newcommand{\crbkcond}[2]{\left( #1 \middle| #2 \right)}

\newcommand{\bmx}[1]{\begin{bmatrix} #1 \end{bmatrix}}

\newcommand{\bd}[1]{\mathbf{#1}}

\definecolor{azure}{rgb}{0.0, 0.4, 0.9}

\definecolor{darkred}{rgb}{0.6, 0, 0}

\definecolor{codegreen}{rgb}{0,0.4,0}
\definecolor{codeblue}{rgb}{0.1,0.1,0.7}
\definecolor{codegray}{rgb}{0.5,0.5,0.5}
\definecolor{codepurple}{rgb}{0.58,0,0.82}
\definecolor{backcolour}{rgb}{0.97,0.97,0.97}
 
\lstdefinestyle{mystyle}{
  backgroundcolor=\color{backcolour},  
  commentstyle=\color{codegreen},
  keywordstyle=\color{magenta},
  numberstyle=\tiny\color{codegray},
  stringstyle=\color{codepurple},
  basicstyle=\scriptsize\ttfamily,
  identifierstyle=\color{codeblue},
  breakatwhitespace=false,     
  breaklines=true,         
  captionpos=b,          
  keepspaces=true,         
  numbers=left,          
  numbersep=4pt,         
  showspaces=false,        
  showstringspaces=false,
  showtabs=true,         
  tabsize=3
}
\numberwithin{equation}{section}

\title{Non-Rectangular Average-Reward Robust MDPs: \\
Optimal Policies and Their Transient Values}

\author[1]{Shengbo Wang}
\affil[1]{Daniel J. Epstein Department of Industrial and Systems Engineering\\
University of Southern California}

\author[2]{Nian Si}
\affil[2]{Department of Industrial Engineering and Decision Analytics\\
Hong Kong University of Science and Technology}
\date{February 2026}

\begin{document}
\maketitle
\begin{abstract}
We study non-rectangular robust Markov decision processes under the average-reward criterion, where the ambiguity set couples transition probabilities across states and the adversary commits to a stationary kernel for the entire horizon. We show that any history-dependent policy achieving sublinear expected regret uniformly over the ambiguity set is robust-optimal, and that the robust value admits a minimax representation as the infimum over the ambiguity set of the classical optimal gains, without requiring any form of rectangularity or robust dynamic programming principle. Under the weak communication assumption, we establish the existence of such policies by converting high-probability regret bounds from the average-reward reinforcement learning literature into the expected-regret criterion. We then introduce a transient-value framework to evaluate finite-time performance of robust optimal policies, proving that average-reward optimality alone can mask arbitrarily poor transients and deriving regret-based lower bounds on transient values. Finally, we construct an epoch-based policy that combines an optimal stationary policy for the worst-case model with an anytime-valid sequential test and an online learning fallback, achieving a constant-order transient value.
\end{abstract}

\section{Introduction}
Robust Markov decision processes (MDPs) provide a foundational framework for reliable sequential decision-making under model ambiguity by requiring the controller to optimize performance against the worst-case transition kernel within a specified ambiguity set. This framework has become increasingly important in contemporary stochastic control, reinforcement learning, and operations contexts, with applications in systems governed by partially specified physical parameters, environments shaped by policy-dependent feedback effects, and data-driven transition models subject to correlated estimation errors. In such settings, the modeler may have difficulty committing to a particular transition dynamic that could be misspecified, making robustness a central design consideration.

The predominant literature on robust MDPs focuses on ambiguity sets that satisfy some form of rectangularity, such as 
SA-rectangularity \citep{iyengar2005robust,gonzalez2002minimax,nilim2005robust},
S-rectangularity \citep{xu2010distributionally,wiesemann2013robust}, 
$r$-rectangularity \citep{Goyal2023Beyond_Rectangularity}, and 
$k$-rectangularity \citep{mannor2016robust}. The primary motivation for imposing rectangularity is tractability: it restores an appropriate dynamic programming principle and reduces the original robust control problem to solving a Bellman-type fixed-point equation. Under rectangularity, nature’s choices decompose across states or state–action pairs, so the adversary’s optimization can be performed locally at each decision epoch.

However, in many data-driven and structurally motivated ambiguity models, the rectangularity assumption is restrictive. Non-rectangular ambiguity sets, in which ambiguity is intrinsically coupled across states, arise naturally in applications: transition probabilities may be linked through shared parameters, joint statistical constraints, or global structural conditions, and thus cannot be perturbed independently on a state-by-state basis. We illustrate this point with two representative examples.

\textbf{Example 1:} (MLE-based confidence regions).
Suppose the transition kernel is estimated from data via maximum likelihood, and ambiguity is modeled through a joint confidence region around the MLE. Such regions are typically defined through likelihood-ratio or score-based constraints that couple all transition probabilities simultaneously. For instance, an asymptotic 
$\chi^2$-type confidence set constrains the entire transition matrix through a single quadratic form. As a result, perturbations to one state–action transition must be compensated elsewhere to remain within the global confidence region. The resulting ambiguity set is inherently non-rectangular. We refer the reader to the construction in \citet{wiesemann2013robust}. 

\textbf{Example 2:} (factorized MDPs). In many applications, transition kernels depend on a set of shared underlying latent factors. For instance, in healthcare settings, patient transitions may be influenced by common factors such as genetic background, demographic characteristics, and physiological attributes associated with particular diseases \cite{grand2023beyond}. In such models, ambiguity arises at the level of these latent factors rather than at individual state–action transitions. 
Importantly, the model ambiguity across factors is typically not independent, and the dependence of transition probabilities on these factors may be nonlinear. As a result, perturbations in one component of the factor space simultaneously affect multiple state–action transitions. The induced ambiguity set is therefore inherently non-rectangular and, in general, does not satisfy 
$r$-rectangularity either.

In contrast to the substantial progress and tractability achieved under rectangular MDPs, the non-rectangular setting is fundamentally more intricate, as the Bellman-type fixed-point equation typically fails to hold. In particular, optimal policies can no longer be expected to be Markovian or to admit a characterization through standard robust Bellman equations. Moreover, it is generally unclear whether the robust value admits any tractable dynamic programming representation.

These issues become even more delicate under the average-reward criterion, which is natural for ongoing systems without an intrinsic horizon \citep{wang2025bellman}. Compared with discounted formulations, average-reward control is already technically more subtle even in the classical, non-robust setting. Optimality conditions and policy structure depend sensitively on the communicating structure of the underlying Markov chain (e.g., multichain, weakly communicating, or unichain models), and the gain–bias decomposition separates steady-state performance from transient effects, introducing additional analytical complexity \citep{puterman2014MDP}. Moreover, long-run optimality alone does not address transient performance, raising the additional question of how to evaluate and control finite-time behavior in such models.
While these challenges are largely understood in the classical (non-robust) theory, they become substantially more delicate in robust models.

This paper addresses these challenges in non-rectangular average-reward robust MDPs. We impose no structural assumptions on the ambiguity set beyond weak communication, which guarantees that, in the non-robust case, the optimal average reward is well defined and independent of the initial state. Recognizing that Markovian policies may fail to be optimal under non-rectangular ambiguity, we adopt the standard formulation in \citet{grand2023beyond}: the controller may employ general history-dependent policies, while the adversary is stationary and commits to a single transition kernel $p\in\mathcal{P}$ for the entire horizon. This non-rectangular model encompasses a broad class of statistically natural ambiguity models, including those induced by joint confidence regions, maximum-likelihood estimators, or nonlinear factor models, while lying outside the classical robust dynamic programming framework that relies on rectangularity to ensure stagewise decomposability.

Under this framework we establish a crucial message: robust optimality under the average-reward criterion is fundamentally linked to learnability. Specifically, we introduce an abstraction of online RL algorithms as a class of history-dependent policies that achieve sublinear regret, and we show that such policies are robust-optimal for (non-)rectangular average-reward problem. This connection offers a different perspective: rather than restoring tractability through structural assumptions on the ambiguity set, robust optimality fundamentally emerges from the ability to achieve online RL over all transition models in the ambiguity set.

While optimal, these online RL policies typically have poor finite-time performance: the transient value (TV) must tend to negative infinity as the horizon increases due to the need for persistent exploration. Motivated by this issue, we go beyond long-run optimality and study finite-time behavior, asking whether one can achieve a uniformly lower bounded TV over any finite horizon. We show that this is indeed possible by proposing a hybrid, epoch-based policy that, perhaps surprisingly, achieves a uniformly lower bounded TV of the same order as the span of the value function.

To obtain this finite-time performance, we design Policy~\ref{policy:O1_TV_policy}, which alternates between exploitation and learning. The policy follows an average-reward optimal stationary policy for a candidate worst-case model while simultaneously running a Markov chain sequential probability ratio test (SPRT) \citep{wald1945sprt}. If the observed data are inconsistent with the model currently being followed, the policy falls back to a standard online RL routine for the remainder of that stage. By scheduling stages to grow over time and making ``false alarms'' increasingly rare, the policy adapts quickly when the model is wrong while preserving robust optimality.

\subsection{Review of Relevant Literature}
\textbf{Rectangular robust MDPs:}  The discounted-reward theory for robust MDPs under rectangular ambiguity sets is now well developed:
Rectangularity restores a dynamic programming principle and leads to Bellman-type optimality equations
\citep{iyengar2005robust,wiesemann2013robust,xu2010distributionally,wang2023foundationRMDP,grand2024tractable,wang2025statistical}.
By contrast, the average-reward setting has received comparatively less attention and has only recently begun to be systematically developed. 
For example, \citet{wang2023avg_unichain_dp} study the average-reward robust problem under
SA-rectangularity together with uniform unichain assumptions. \citet{grand2023beyond} investigate
Blackwell optimality, proving that $\epsilon$-Blackwell optimal policies exist under SA-rectangularity,
and highlighting important pathologies beyond this regime: in S-rectangular robust MDPs, average-reward optimal policies may fail to exist. Relatedly, \citet{grand2025playing} show that for 
$s$-rectangular robust MDPs, Blackwell 
$\epsilon$-optimal policies may fail to be Markovian, while full history dependence is unavoidable in general. More recently, \citet{wang2025bellman} develops dynamic programming and Bellman equations for average-reward formulations under broad conditions within the rectangular framework. Taken together, these works highlight both the analytical strength of rectangularity and the fragility of average-reward optimality, thereby motivating our study of non-rectangular ambiguity under a stationary committing adversary in the average-reward setting.

\textbf{Non-rectangular robust MDPs:}  Recent progress on robust MDPs beyond rectangularity has focused on algorithmic methods that bypass the breakdown of robust Bellman recursions under coupled ambiguity. \citet{li2026policy} develop policy-gradient algorithms for robust infinite-horizon MDPs with non-rectangular transition ambiguity sets, including methods that achieve globally optimal robust policy evaluation, albeit at significant computational cost. For the average-reward criterion, \citet{wang2025provable} propose a projected policy-gradient method for ergodic tabular robust average-reward MDPs with general (non-rectangular) ambiguity and analyze convergence of both the policy updates and the inner worst-case transition evaluation. Complementing gradient-based approaches, \citet{kumar2025dual,kumar2025non} identify a tractable family of coupled $L_p$-bounded transition uncertainty sets, derive a dual formulation, and obtain polynomial-time robust policy evaluation for the $L_1$ case. More recently, \citet{satheesh2026provably} provide provable efficiency guarantees for robust MDPs with general policy parameterizations under  non-rectangular uncertainty by reducing average-reward problems to entropy-regularized discounted formulations amenable to first-order methods. In a related but distinct setting, \citet{gadot2024solving} study non-rectangular \emph{reward}-robust MDPs with known dynamics, showing that coupled reward ambiguity induces a visitation-frequency regularizer and enabling convergent policy-gradient learning. In contrast to the above work, which primarily targets computational methods and convergence guarantees for computing robust policies, our paper takes a structural viewpoint in the stationary-adversary, non-rectangular average-reward setting: we characterize robust optimality through the existence of online-RL policies, culminating in a robust-optimal policy with constant-order transient value via sequential testing.

\textbf{Online RL for average-reward MDPs:} The online RL literature for average-reward MDPs has developed along several threads. \citet{jaksch2010near} introduced the UCRL2 algorithm, which achieves $\widetilde{O}(\sqrt{T})$ regret under the diameter assumption. For the more general weakly communicating setting, \citet{bartlett2009WC} proposed REGAL, attaining $\widetilde{O}(\sqrt{T})$ regret with dependence on the span of the optimal bias vector rather than the diameter. \citet{fruit2018efficient} and \citet{wei2020model} developed algorithms with improved span-dependent bounds. More recently, \citet{zhang2023amdp_regret} provided a model-free algorithm (UCB-AVG) with near-optimal high-probability regret bounds under weak communication. Our work leverages these advances by showing that any such algorithm with sublinear regret can be converted into a robust-optimal policy for non-rectangular average-reward robust MDPs.

\subsection{Contributions and Organization}
Our contributions and main results are organized as follows. 

First, we connect online average-reward RL to the non-rectangular robust average-reward MDP. In Theorem~\ref{thm:value_RL_policy}, we show that any online RL algorithm with uniform sublinear expected or high-probability regret can be converted into a robust optimal policy.

We also clarify that such policies need not exist in full generality. Without additional structural conditions on the ambiguity set, we construct an instance in Proposition~\ref{prop:no_sublinear_regret} in which every policy incurs linear regret. Nevertheless, Proposition~\ref{prop:hp_RL_policy} shows that under weak communication, standard high-probability regret guarantees for average-reward RL can be converted into an expected regret bound, which in turn implies the existence of robust optimal online RL policies.

We then study finite-time behavior through a new performance notion, the transient value (TV), which can be viewed as a refined ``negative regret'' term measuring cumulative deviation from the robust average reward. In Proposition~\ref{prop:TL_upper_bd}, we establish a general upper bound on TV in terms of the span of worst-case bias functions, showing that one cannot accumulate an indefinitely growing expected deviation from the optimal robust average reward. We further relate transient performance to learning speed: for an RL policy with a given regret rate, Proposition~\ref{prop:TV_lower_bd_RL_policy} provides a corresponding lower bound on its transient value, making explicit how typical regret growth (e.g., $\sqrt{T}$) translates into the corresponding TV scaling (e.g., $-\sqrt{T}$).

Finally, Section~\ref{sec:construct_policy} gives an explicit construction of a robust-optimal policy with TV uniformly lower bounded by the span of the worst-case bias function. The policy operates in epochs: it executes a candidate worst-case stationary policy while simultaneously running an anytime-valid Markovian composite sequential probability ratio test (SPRT) \citep{wald1945sprt}. Upon rejection, it switches to a learning-based procedure within the same epoch. In Theorem~\ref{thm:pistar_tv_bound}, we show that this design yields a constant-order lower bound on TV while maintaining robust optimality.

A key technical ingredient behind this result is the design and analysis of the composite SPRT for Markov chains, with carefully scheduled type-I error control and logarithmic expected detection delay under alternatives (Theorem~\ref{thm:rej_time}), which allows the policy to adapt quickly when the current model is misspecified. The second ingredient is the scheduling and use of the online RL policy after rejection (whether accidental or correct), ensuring that the controller can adequately explore and exploit when the test indicates a model mismatch.

\section{Canonical Construction and the Optimal Robust Control Problem}\label{sec:canonical_construction}
In this section, we present a brief but self-contained canonical construction of the probability space, the processes of interest, and the controller's and adversary's policy classes. We adopt the notation and framework of \citet{wang2023foundationRMDP}, to which we refer the reader for additional details.

Let $S,A$ be finite state and action spaces, each equipped with the discrete Borel $\sigma$-fields $\cS$ and $\cA$, respectively. Define the underlying measurable space $(\Omega, \cF)$ with $\Omega = (S\times A)^{\Z_{\geq 0}}$ and $\cF$ the corresponding cylinder $\sigma$-field. The process $\{(X_t,A_t), t \geq 0\}$ is defined by point evaluation, i.e., $X_t(\omega) = s_t$ and $A_t(\omega) = a_t$ for all $t \geq 0$ and any $\omega = (s_0,a_0,s_1,a_1,\ldots) \in \Omega$.

The history set $\bd{H}_t$ at time $t$ contains all $t$-truncated sample paths $$\bd{H}_t := \set{h_t = (s_0,a_0,\ds,a_{t-1} ,s_t): \omega = (s_0,a_0,s_1\ds )\in\Omega}.$$
We also define the random element $H_t:\Omega\ra \bd{H}_t$ by point evaluation $H_t(\omega) = h_t$, and the $\sigma$-field $\cH_t:=\sigma(H_t)$.

Motivated by applications in which the controller must be restricted to a subset of randomized decisions (e.g., to deterministic actions), we consider the constraint set as a prescribed subset $\cQ \subseteq \cP(\cA)$. Let $\delta_a\in\cP(\cA)$ denote the point mass measure at $a\in A$. For the compatibility with the RL literature, we will frequently require that $\set{\delta_a:a\in A }\subseteq  \cQ$.

A controller policy $\pi$ is a sequence of decision rules $\pi = (\pi_0,\pi_1,\pi_2,\ldots)$ where each $\pi_t$ is a measure-valued function $\pi_t:\bd{H}_t \ra \cQ$, represented in conditional distribution form as $\pi_t(a|h_t) \in [0,1]$ with $\sum_{a \in A}\pi_t(a|h_t) = 1$. The history-dependent controller policy class is therefore $$\PiH(\cQ) := \{ \pi = (\pi_0,\pi_1,\ldots) : \pi_t \in \{\bd{H}_t \ra \cQ\}, \ \forall t \geq 0 \}.$$

A controller policy $\pi = (\pi_0,\pi_1,\ldots)$ is stationary if for any $t_1,t_2 \geq 0$ and $h_{t_1} \in \bd{H}_{t_1}, h'_{t_2} \in \bd{H}_{t_2}$ such that $s_{t_1} = s'_{t_2}$, we have $\pi_{t_1}(\cdot|h_{t_1}) = \pi_{t_2}(\cdot|h'_{t_2})$. In particular, this means $\pi_t(a|h_{t}) = \Delta(a|s_t)$ where $h_t =(s_0,a_0,\ds, s_t) $ for some $\Delta:S \to \cQ$ for all $t \geq 0$. Thus, a stationary controller policy can be identified with $\Delta:S \to \cQ$, i.e., $\pi = (\Delta,\Delta,\ldots)$. Accordingly, the stationary policy class for the controller is $$\PiS(\cQ) := \{ (\Delta,\Delta,\ldots) : \Delta \in \{S \to \cQ\} \},$$ which is identified with $\{S \to \cQ\}$.

In this paper, we consider a general non-rectangular stationary adversary who, similar to the controller side, chooses a fixed (over the entire control horizon) transition kernel from a prescribed subset $\cP\subseteq  \set{S\times A\ra\cP(\cS)}$. Here, non-rectangularity means that no further factorization constraints are imposed on $\cP$. For instance, popular rectangular models in the literature include S-rectangularity--$\cP=\bigtimes_{s\in S}\cP_{s}$ with $\cP_{s}\subseteq  \set{A\ra\cP(\cS)}$ \citep{wiesemann2013robust,xu2010distributionally}--and SA-rectangularity--$\cP=\bigtimes_{s\in S}\bigtimes_{a\in A}\cP_{s,a}$ with $\cP_{s,a}\subseteq  \cP(\cS)$ \citep{iyengar2005robust,nilim2005robust}.

Given an initial distribution $\mu\in\cP(\cS)$, a controller policy $\pi\in\PiH$, and an adversarial kernel $p\in\cP$, the probability measure $P_{\mu}^{\pi,p}$ on $(\Omega,\cF)$ is uniquely determined by
\[
P_{\mu}^{\pi,p}(H_t=h_t)
=\mu(s_0)\cd\pi_0(a_0| h_0)p(s_1| s_0,a_0)\cdots
\pi_{t-1}(a_{t-1}| h_{t-1})p(s_t| s_{t-1},a_{t-1})
\]
for all $t$ and all histories $h_t=(s_0,a_0,\ldots,s_t)\in\bd{H}_t$.

We define the upper and lower long-run average rewards associated with $(\mu,\pi,p)$ by
\[
\overline{\alpha}(\mu,\pi,p)
:= \limsup_{n\to\infty} E_\mu^{\pi,p}\left[\frac{1}{n}\sum_{k=0}^{n-1} r(X_k,A_k)\right],
\qquad
\underline{\alpha}(\mu,\pi,p)
:= \liminf_{n\to\infty} E_\mu^{\pi,p}\left[\frac{1}{n}\sum_{k=0}^{n-1} r(X_k,A_k)\right].
\]
In general, $\limsup$ and $\liminf$ need not coincide when $\pi$ is history-dependent.

This paper studies the optimal robust control of these criteria for a robust MDP instance $(\cQ,\cP,r)$. The upper and lower optimal robust control values are defined by
\[
\overline{\alpha}(\mu,\Pi,\cP)
:= \sup_{\pi\in\Pi}\inf_{p\in\cP}\overline{\alpha}(\mu,\pi,p),
\qquad
\underline{\alpha}(\mu,\Pi,\cP)
:= \sup_{\pi\in\Pi}\inf_{p\in\cP}\underline{\alpha}(\mu,\pi,p),
\]
where the controller's policy class $\Pi$ is either $\PiH(\cQ)$ or $\PiS(\cQ)$.

\subsection{Dynamic Programming for Classical Average-Reward MDPs}

In this section, we review the dynamic programming properties satisfied by non-robust average-reward MDPs. In the classical MDP setting, the adversary is absent, and for each controlled transition kernel $p$ the upper and lower optimal control value is given by 
$\sup_{\pi\in\PiH(\cQ)}\overline{\alpha}(\mu,\pi,p)$ and $\sup_{\pi\in\PiH(\cQ)} \underline{\alpha}(\mu,\pi,p).$

Since it is assumed that $\set{\delta_a:a\in A}\subseteq   \cQ$, the dynamic programming principle for average-reward MDP (see, e.g., \citet[Theorem 9.1.4 and 9.1.7]{puterman2014MDP}) implies that \begin{equation}\label{eqn:classical_amdp_dpp}
\sup_{\pi \in \PiH(\cQ)}\underline{\alpha}(\mu,\pi,p) = \sup_{\pi \in\PiS(\cQ)}\underline{\alpha}(\mu,\pi,p) = \sum_{s\in S}\mu(s)\alpha_p(s),
\end{equation} 
and the same holds for $\overline\alpha$. Here $\alpha_p: S\ra\R$ is part of the solution $(\alpha_p,v_p)$ to the \textit{modified} (see the discussion in \citet{puterman2014MDP} comparing equations 9.1.2 and 9.1.8) multichain Bellman equation
\begin{equation}\label{eqn:multichain_eqn}\begin{aligned}
  \alpha_p(s) &= \max_{a\in A}\sum_{s'\in S}p(s'|s,a)\alpha_p(s') \\
  v_p(s)&= \max_{a\in A}\set{r(s,a) - \alpha_p(s) + \sum_{s'\in S}p(s'|s,a)v_p(s')}
\end{aligned}
\end{equation}
for all $s\in S$.
A solution to \eqref{eqn:multichain_eqn} always exists when the state and action spaces are finite \citep[Theorem 9.1.4 and Proposition 9.1.1]{puterman2014MDP}. Moreover, in this setting, $\alpha_p$ is unique in the sense that if $(\alpha',v')$ also solves \eqref{eqn:multichain_eqn}, then $\alpha_p = \alpha'$.

\section{Robust Optimality of Online RL Policies}

\subsection{Online RL Algorithms and Expected Regret}
In this section, we define regret under the average-reward criterion and provide an abstraction of online RL algorithms viewed as history-dependent policies.

Regret for an online RL policy measures the deviation of a policy’s realized value from the optimal long-run average reward over a finite horizon $T$. Define
\begin{equation}\label{eqn:regret_def}
R_{r,p}(T) := \sum_{t=0}^{T-1}\sqbk{\alpha_p(X_0) - r(X_t,A_t)},
\end{equation}

where $\alpha_p$ is the unique solution to \eqref{eqn:multichain_eqn}. Then the expected regret of a policy $\pi$ is given by $E_\mu ^ {\pi,p}R_{r,p}(T)$. Because $S$ and $A$ are finite, any history-dependent policy has expected regret $E_\mu^{\pi,p}R_{r,p}(T) = O(T)$ as $T\to\infty$. 

An online RL algorithm can be interpreted as a systematic procedure that selects (possibly randomized) actions based on the observed history of states and actions and thus induces a history-dependent policy $\pi\in\PiH$. The key distinction between a reasonable online RL procedure and an arbitrary history-dependent policy is \emph{sublinear expected regret}. Intuitively, achieving sublinear expected regret requires the induced policy to approximate optimal decisions with arbitrarily small error as $T\to\infty$. This motivates the following definition.
\begin{definition}[Online RL Policies]\label{def:online_RL}
We say that a policy $\piRL\in\PiH(\cQ)$ achieves online RL over the class of transition kernels $\cP\subseteq   \set{S\times A\ra\cP(S)}$ if for all $p\in\cP$, $\mu\in \cP(S)$, and $r:S\times A\ra [0,1]$, 
$$\limsup_{T\ra \infty} \frac{1}{T }E_{\mu}^{\piRL,p}R_{r,p}(T) = 0.$$
We denote policies $\pi\in\PiH(\cQ)$ that achieve online RL over $\cP$ by $\PiRL(\cQ,\cP)$.
\end{definition}
We note that in the literature, online RL algorithms often assume a fixed initial distribution, such as the uniform distribution or a specific starting state. However, most approaches readily extend to arbitrary initial distributions. 

\subsection{Optimality of Online RL Policies}
We define the robust optimal average reward as the worst-case (over the ambiguity set) classical optimal gain, weighted by the initial distribution:
\begin{equation}\label{eqn:alpha_star_mu}
 \alpha^*(\mu):=\inf_{p\in\cP}\sum_{s\in S}\mu(s)\alpha_p(s),\quad\text{and}\quad \alpha^*(s):= \alpha^*(\delta_s)
\end{equation}
for all $s\in S$, where $\delta_s\in\cP(\cS)$ is the point mass measure at $s$.
With these definitions, we are ready to present a main result.

\begin{theorem}\label{thm:value_RL_policy}
Assume that there exists $\piRL\in \PiRL(\cQ,\cP)$. Then, for all $\mu\in\cP(S)$, $$\begin{aligned}
  \underline{\alpha}(\mu,\PiH(\cQ),\cP) &= \inf_{p\in\cP}\underline{\alpha}(\mu,\piRL,p)\\
  &=\inf_{p\in\cP}\sup_{\pi\in\PiH(\cQ)}\underline\alpha(\mu,\pi,p) = \inf_{p\in\cP}\sup_{\pi\in\PiS(\cQ)}\underline\alpha(\mu,\pi,p) \\
  &= \alpha^*(\mu).
\end{aligned}$$
The same result holds true if $\underline{\alpha}$ is replaced with $\overline{\alpha}$. 
\end{theorem}

\begin{proof}[Proof of Theorem \ref{thm:value_RL_policy}]

From the dynamic programming theory for classical MDPs, we have the characterization \eqref{eqn:classical_amdp_dpp} where \[
\sup_{\pi\in\PiH(\cQ)}\underline{\alpha}(\mu,\pi,p) = \sup_{\pi\in\PiS(\cQ)}\underline{\alpha}(\mu,\pi,p) = \sum_{s\in S}\mu(s)\alpha_p(s),
\]
and the same holds for $\overline\alpha$. Taking the infimum over $p\in\cP$, we get the equality of the last three quantities in Theorem \ref{thm:value_RL_policy}.

Moreover, by weak duality and policy class inclusion, it is not hard to see that
$$\inf_{p\in\cP}\underline{\alpha}(\mu,\piRL,p)\leq \sup_{\pi\in\PiH(\cQ)}\inf_{p\in\cP}\underline{\alpha}(\mu,\pi,p) = \underline{\alpha}(\mu,\PiH(\cQ),\cP) \leq \inf_{p\in\cP}\sup_{\PiH(\cQ)}\underline{\alpha}(\mu,\pi,p),$$
with the same inequalities holding for $\overline{\alpha}$.

Therefore, if we can show that for all $p\in\cP$, 
\begin{equation}\label{eqn:to_show_RL_policy_val}
  \underline{\alpha}(\mu,\piRL,p) = \overline{\alpha}(\mu,\piRL,p) = \sum_{s\in S}\mu(s)\alpha_p(s), 
\end{equation}
then the conclusion of Theorem \ref{thm:value_RL_policy} would follow. 

To show \eqref{eqn:to_show_RL_policy_val}, we recall that by Definition \ref{def:online_RL}, $\piRL\in\PiRL(\cQ,\cP)$ satisfies that 
$$\begin{aligned}0 &= \limsup_{T\ra \infty} \frac{1}{T }E_{\mu}^{\piRL,p}R_{r,p}(T) \\
&=\sum_{s\in S}\mu(s)\alpha_p(s)-\liminf_{T\ra\infty} E_\mu^{\piRL,p}\sqbk{\frac{1}{T}\sum_{t=0}^{T-1}r(X_t,A_t)}\\
&= \sum_{s\in S}\mu(s)\alpha_p(s) - \underline{\alpha}(\mu,\piRL,p)
\end{aligned}$$
for all $p\in\cP$, $\mu\in \cP(S)$, and $r:S\times A\ra [0,1]$. 

Moreover, combine this with \eqref{eqn:classical_amdp_dpp}, we conclude that
$$\begin{aligned}\sum_{s\in S}\mu(s)\alpha_p(s) &= \underline{\alpha}(\mu,\piRL,p) \\
&\leq \overline{\alpha}(\mu,\piRL,p)\\
&\leq\sup_{\pi\in\PiH(\cQ)}\overline\alpha(\mu,\pi,p) \\
&= \sum_{s\in S}\mu(s)\alpha_p(s).
\end{aligned}$$
This implies \eqref{eqn:to_show_RL_policy_val}, hence completes the proof.
\end{proof}

\subsection{Weak Communication and the Existence of Online RL Policies}\label{sec:WC_and_online_RL}

While algorithms for online RL under the general gain setting remain (justifiably, to be discussed in Example \ref{example:no_online_RL_alg}) underexplored, there has been substantial interest and progress in the RL literature under various notions of communicating structures. Among these, one of the most general settings that has been actively studied and has yielded several positive results is that of weakly communicating MDPs.

In this section, we first present a simple example showing that, without additional assumptions on $\cP$, it is possible that $\PiRL(\cQ,\cP) = \varnothing$. In such a case, Theorem~\ref{thm:value_RL_policy} becomes vacuous. However, we demonstrate that Theorem~\ref{thm:value_RL_policy} remains meaningful--namely, that there exists $\piRL \in \PiRL(\cQ,\cP)$--when we restrict attention to weakly communicating $\cP$ as defined in Definition~\ref{def:wc}. To establish the existence of such a $\piRL$, we introduce a principled method to convert almost any online RL algorithm for weakly communicating MDPs that satisfies a sublinear high-probability regret upper bound into a policy that achieves online RL in the sense of Definition~\ref{def:online_RL}.

\begin{example}\label{example:no_online_RL_alg}
Consider the following setting with 3 states $\set{i,e_1,e_2}$ (one initial state and two ending states that are absorbing) and 2 actions $\set{a_1,a_2}$. We consider $\cP = \set{p^{(1)},p^{(2)}}$ with
$$p^{(1)}(\cd|\cd,a_1)=\bmx{0 &1&0\\0&1&0\\0&0&1 },\qquad p^{(1)}(\cd|\cd,a_2)=\bmx{0 &0&1\\0&1&0\\0&0&1 },$$ $p^{(2)}(\cd|\cd,a_1) = p^{(1)}(\cd|\cd,a_2)$, and $p^{(2)}(\cd|\cd,a_2) = p^{(1)}(\cd|\cd,a_1)$. So, both $e_1$ and $e_2$ are absorbing states under $p^{(i)}$, $i=1,2$ regardless which action is taken. The reward function is defined by $r(i,\cd) = 0$, $r(e_1,\cd) =1$, and $r(e_2,\cd) =0$. 

We consider any policy $\pi:=(\pi_0,\pi_1,\ds)\in\PiH(\cQ)$. Choose the initial distribution as the point mass measure at state $i$ and write $E_i^{\pi,p}$ as a shorthand for $E_{\delta_i}^{\pi,p}$. Let $b:=\min \set{\pi_0(a_1|i),\pi_0(a_2|i)}$. In this setting, we show that the following Proposition holds 
\begin{proposition}\label{prop:no_sublinear_regret}
For any $\cQ\subseteq  \cP(\cA)$, the regret for any policy $\pi\in\PiH(\cQ)$ satisfies 
$$\max_{p\in\cP}\limsup_{T\ra\infty}\frac{1}{T}E_i^{\pi,p}R_{r,p}(T) = 1-b \geq\frac{1}{2}.$$
In particular, $\PiRL(\cQ,\cP) = \varnothing$. 
\end{proposition}
\begin{proof}[Proof of Proposition \ref{prop:no_sublinear_regret}]
First, note that for $k = 1,2$, under kernel $p^{(k)}$ and starting from state $i$, taking action $a_k$ always leads to $e_1$ as the next state. Since $e_1$ is an absorbing state with reward $1$, the long-run average reward under this action is $1$. Consequently, $\alpha_p(i) = 1$ for all $p \in \cP$.

So, we have that for $k = 1,2$
\begin{align*}\frac{1}{T}E_i^{\pi,p^{(k)}}R_{r,p}(T) &= 1 - E^{\pi,p^{(k)}}_i\sqbk{\frac{1}{T}\sum_{t=0}^{T-1} r(X_t,A_t)} \\
&= 1-E_i^{\pi,p^{(k)}}\sqbk{\frac{T-1}{T}\1\set{X_1 = e_1}}\\
&= 1-\frac{T-1}{T}\pi_0(a_k|i)
\end{align*}
where the last equality follows from the observation under kernel $p^{(k)}$ and starting from state $i$, taking action $a_k$ always leads to $e_1$ as the next state; but taking action $a_{(k\text{ mod }2)+1}$ always leads to $e_2$. Taking $T\ra\infty$ and $\max$ over $k = 1,2$ and noting that $b\leq 1/2$, we obtain the statement of the proposition. 
\end{proof}
\end{example}

While Example~\ref{example:no_online_RL_alg} shows a negative result in which no policy can achieve online RL, this failure arises because state $i$ is visited only once, regardless of the controller’s policy. Consequently, the controller has no opportunity to learn the correct action at $i$ before the process is absorbed into either $e_1$ or $e_2$. 

In contrast, if an MDP possesses some form of ``irreducibility"--in particular, if all states can be visited infinitely often under some control policy--then achieving online RL becomes feasible. This observation motivates the consideration of a weak form of irreducibility, known as weakly communicating in the classical MDP and RL literature, which serves as a sufficient condition to guarantee the learnability of the optimal policy.

We proceed by first introducing the following notation. For $p\in \cP$ and $\Delta:S\ra\cP(\cA)$, denote
$$p_\Delta(s'|s):= \sum_{a\in A} p(s'|s,a)\Delta(a|s).$$ Also, let $p_{\Delta}^n(s'|s)$ be the $(s,s')$ entry of the $n$-th power of the matrix $\set{p_\Delta(s'|s):s,s'\in S}$. Moreover, for $C\subseteq S$ we denote the complement of $C$ is $S$ by $C^c:= S\setminus C$. 

\begin{definition}[Weak Communication]\label{def:wc}
Consider arbitrary controller and adversary action sets $\cQ$ and $\cP$. A transition kernel $p\in\cP$ is said to be weakly communicating if there is a communicating class $C_p\subseteq S$ s.t. for any $s,s'\in C_p$, there exists $\Delta:S\ra\cQ$ and $N\geq 1$ s.t. $p_\Delta^N(s'|s) > 0$. Moreover, for all $s\in C_p^c$, $s$ is transient under any stationary controller policy. 
   
The ambiguity set $\cP$ is weakly communicating if every $p \in \cP$ is weakly communicating. 

\end{definition}

Weak communication has been a standard assumption in the classical MDP setting that guarantees the optimal average reward is initial state independent \citep{puterman2014MDP}. In particular, if $\cP$ is weakly communicating and $\set{\delta_a:a\in A}\subseteq\cQ$, then for all $p\in \cP$, the classical dynamic programming theory simplifies (cf. \eqref{eqn:classical_amdp_dpp}) to 
\begin{equation}\label{eqn:non_rob_dpp}\sup_{\pi\in\PiH(\cQ)}\overline{\alpha} (\mu,\pi,p) = \sup_{\pi\in\PiS(\cQ)}\overline{\alpha} (\mu,\pi,p)= \alpha_p,
\end{equation}
where $\alpha_p\in[0,1]$ is a part of a solution pair $(\alpha_p,v_p)$ to the constant-gain Bellman equation (cf. \eqref{eqn:multichain_eqn})
\begin{equation}\label{eqn:wc_Bellman_eqn}
  v_p(s) =\max_{a\in A} \set{r(s,a) - \alpha_p + \sum_{s'\in S}p(s'|s,a)v_p(s')},
\end{equation}
for all $s\in S$. Moreover, $\alpha_p$ is unique in the sense that if $(\alpha',v')$ also solves \eqref{eqn:wc_Bellman_eqn}, then $\alpha_p = \alpha'$. But in general, $v_p$ could be non-unique. 

In this setting, since $\alpha_p$ is state-independent, the definition of $\alpha^*(\mu)$ in \eqref{eqn:alpha_star_mu} also simplifies. In particular, we write $\alpha^* = \inf_{p\in\cP}\alpha_p$.

With this definition of weakly communicating $\cP$, we now discuss the existence of policies that achieve online RL within weakly communicating models. There is already a substantial body of work on online RL algorithms for the average-reward setting \citep{jaksch2010near,bartlett2009WC,fruit2018efficient,zhang2023amdp_regret}. However, most of this literature under weak communication has focused on establishing \emph{high-probability} regret bounds rather than expected regret, as defined in Definition~\ref{def:online_RL}. To connect these results with our framework, we show in the following proposition that a policy satisfying the online RL criterion in Definition~\ref{def:online_RL} can be constructed from algorithms that guarantee high-probability regret bounds.

\begin{proposition}\label{prop:hp_RL_policy}
Suppose that $\cP$ is weakly communicating. Assume further that there exist a sequence of policies $\{\pi(T): T \ge 0\} \subseteq   \PiH(\cQ)$ with the following property. There exist parameters $\eta \in (0,1)$, $K \ge 1$, and $\beta > 0$ such that for all $p \in \cP$, reward function $r:S \times A \to [0,1]$, initial distribution $\mu \in \cP(S)$, and $T \ge K$, there exists 
$f_{\mu,r,p}\geq 0$ such that
\[
P_{\mu}^{\pi(T),p} \left( \frac{R_{r,p}(T)}{f_{\mu,r,p} T^\eta} \le 1 \right)
\ge 1 - \frac{1}{T^\beta}.
\]
Then, there exists a policy $\pi \in \PiRL(\cQ,\cP)$.

\end{proposition}

\begin{proof}[Proof of Proposition \ref{prop:hp_RL_policy}]

We construct a policy $\pi\in\PiRL(\cQ,\cP)$ using an \emph{epoch-based} procedure. 
The $n$th epoch spans time steps $t = 2^nK, \ldots, 2^{n+1}K - 1$. 
At the beginning of each epoch, all previously collected data are discarded, and a fresh instance of the online RL algorithm $\pi(2^nK)$ is executed for the duration of that epoch. 
This construction ensures that the regret accumulated across epochs can be bounded in terms of the regret guarantees of $\pi(T)$ provided in the assumption of Proposition \ref{prop:hp_RL_policy}. 

We start by bounding the expected regret within one epoch. We first note that since $\cP$ is weakly communicating, by \eqref{eqn:non_rob_dpp} and \eqref{eqn:wc_Bellman_eqn}, $\alpha_p$ is state independent. By the high-probability bound assumption in Proposition \ref{prop:hp_RL_policy} and the fact that $R_{r,p}(T)\le T$ almost surely, we have
\[
E_{\mu}^{\pi(T),p} R_{r,p}(T)
  \le \Bigl(1-\frac{1}{T^\beta}\Bigr)T^\eta f_{\mu,r,p} + T^{1-\beta}
  \le T^{\max\{\eta, 1-\beta\}}\bigl(f_{\mu,r,p} + 1\bigr),
\]
for any $\mu\in\cP(S)$, $r:S\times A\to[0,1]$, and weakly communicating $p$. 
This inequality bounds the regret of $\pi(T)$ over any horizon of length $T$.

Let $\mu_{2^nK}$ denote the distribution of $X_{2^nK}$ under $E_\mu^{\pi,p}$. 
Since each epoch begins by resetting the history, the expected regret accumulated during the $n$th epoch can be expressed as
\[
\begin{aligned}
E_{\mu}^{\pi,p}\!\left[\sum_{t=2^nK}^{2^{n+1}K-1}\!\!\bigl(\alpha_p - r(X_t,A_t)\bigr)\right]
  &= E_{\mu_{2^nK}}^{\pi(2^nK)}\!\left[\sum_{t=0}^{2^nK-1}\!\!\bigl(\alpha_p - r(X_t,A_t)\bigr)\right] \\
  &\le (f_{\mu,r,p} + 1)(2^nK)^{\max\{\eta, 1-\beta\}} \\
  &\le (f_{\mu,r,p} + 1)2^{n\max\{\eta, 1-\beta\}}K.
\end{aligned}
\]
This bound follows from substituting $T = 2^nK$ into the previous inequality and noting that $\eta < 1$ and $\beta > 0$, giving us the expected regret bound within epoch $n$. 

We first establish the bound at epoch boundaries.  Summing the epoch regret bounds over completed epochs $n=0,\ldots,j-1$, the expected regret at the epoch boundary $T' = 2^jK$ satisfies
\begin{equation}\label{eqn:to_use_total_regret_bd_at_epoch}
E_\mu^{\pi,p} R_{r,p}(T')
  \le \sum_{0\le n\le j-1}(f_{\mu,r,p}+1)2^{n\max\{\eta, 1-\beta\}}K
  \le (f_{\mu,r,p}+1)\frac{2^{j\max\{\eta, 1-\beta\}}-1}{2^{\max\{\eta, 1-\beta\}}-1}K.
\end{equation}

Next, we extend this bound to an arbitrary horizon $T$.  For any $T$ with $2^{j-1}K \le T \le 2^jK$, let $T' = 2^jK$ be the next epoch boundary.  By the average-reward optimality equation \eqref{eqn:wc_Bellman_eqn}, for every action~$a$ and state~$s$,
$r(s,a) - \alpha_p \le v_p^*(s) - \sum_{s'}p(s'|s,a)v_p^*(s')$.
Taking expectations and summing from $t = T$ to $T'-1$, a telescoping argument yields
\[
E_\mu^{\pi,p}\!\left[\sum_{t=T}^{T'-1}\bigl(\alpha_p - r(X_t,A_t)\bigr)\right] \ge -\spnorm{v_p^*}.
\]
Therefore,
\[
E_\mu^{\pi,p} R_{r,p}(T) \le E_\mu^{\pi,p} R_{r,p}(T') + \spnorm{v_p^*}.
\]

Since $\eta>0$ and $1-\beta<1$, define a finite normalization constant
\[
f'_{\mu,r,p} := \frac{4}{2^{\max\{\eta, 1-\beta\}}-1}\bigl(f_{\mu,r,p}+1\bigr) < \infty.
\]
Combining the epoch-boundary bound~\eqref{eqn:to_use_total_regret_bd_at_epoch} with the extension above, for any $2^{j-1}K \le T \le 2^jK$,
\[
\frac{1}{T }E_\mu^{\pi,p} R_{r,p}(T)
  \le f'_{\mu,r,p} 2^{j\max\{\eta-1, -\beta\}} + \frac{\spnorm{v_p^*}}{T},
\]
which converges to $0$ as $j\to\infty$ because $\eta < 1$ and $\beta>0$. This implies that $\pi\in\PiRL(\cQ,\cP)$.
\end{proof}

\begin{remark*}
  The bound in Proposition \ref{prop:hp_RL_policy} can be refined by incorporating logarithmic terms following $T^\eta$. However, since this paper primarily focuses on policy optimality rather than regret analysis, we omit this generalization for the clarity of the current presentation.
\end{remark*}

We provide the following Example \ref{example:RL_alg_ZX} where Proposition \ref{prop:hp_RL_policy} is applied to turn an RL algorithm (viewed as a sequence of history-dependent policies $\set{\pi(T):T\geq 0}$) in \citet{zhang2023amdp_regret} with high-probability regret bound to an online RL policy $\piRL\in\PiRL(\cQ,\cP)$ under Definition \ref{def:online_RL}. 

\begin{example}[The UCB-AVG Algorithm in \citet{zhang2023amdp_regret}]\label{example:RL_alg_ZX}
Suppose $\set{\delta_a:a\in A}\subseteq  \cQ$ and $\cP$ is weakly communicating. Then \citet[Theorem 1]{zhang2023amdp_regret} shows that the UCB-AVG algorithm, identified as $\pi_{\text{ZX}}^\delta\in \PiH(\cQ)$, which only uses deterministic actions, achieves
$$P_\mu^{\pi_{\text{ZX}}^\delta,p}\crbk{R_{r,p}(T) \leq C\spnorm{v_p^*} \sqrt{T}\log\crbk{ \frac{T}{\delta}} }\geq 1-\delta$$ 
for all initial distributions $\mu$ and all $T\geq 1$. Here $C$ is a constant that only depends on $|S|$ and $|A|$, and $v_p^*:S\ra\R$ is one of the solutions to \eqref{eqn:wc_Bellman_eqn}, which depends on $r$ and $p$. 

Choosing $\delta = 1/\sqrt{T}$ and writing $\pi_{\text{ZX}}(T):=\pi_{\text{ZX}}^{1/\sqrt{T}}$, we have that 
$$P_\mu^{\pi_{\text{ZX}}(T),p}\crbk{R_{r,p}(T) \leq 2C\spnorm{v_p^*} \sqrt{T}\log T }\geq 1-\frac{1}{T^{1/2}}.$$ 

Then, we choose $\beta = 1/2$ and $\eta = 1/2+\epsilon$ for any $0 < \epsilon <1/2$. Hence, we see that $\set{\pi_{\text{ZX}}(T):T\geq 0}$ satisfies Proposition \ref{prop:hp_RL_policy}, with $f_{\mu,r,p} = \spnorm{v_p^*} < \infty$ and any $K$ large enough so that $$\frac{2 C\sqrt{K}\log K}{K^{1/2+\epsilon}} \leq 1.$$ 
Hence, by Proposition \ref{prop:hp_RL_policy}, one can construct $\pi_{\mathrm{ZX}}\in\PiRL(\cQ,\cP)$. This shows the following corollary. 

\begin{corollary}\label{cor:piZX_RL}Suppose that $\cP$ is weakly communicating and $\set{\delta_a:a\in A}\subseteq   \cQ$, then $\pi_{\mathrm{ZX}}\in\PiRL(\cQ,\cP)$ which achieves the optimal robust average reward
\[
\underline{\alpha}(\mu,\PiH(\cQ),\cP) = \inf_{p\in\cP}\underline{\alpha}(\mu,\pi_{\mathrm{ZX}},p) = \inf_{p\in\cP}\alpha_p,
\]
where $\alpha_p$ is the unique solution to \eqref{eqn:wc_Bellman_eqn}. 
  
\end{corollary}

\end{example}

\section{Transient Values for Optimal Robust Policies}

By Theorem \ref{thm:value_RL_policy}, if $\PiRL(\cQ,\cP) \neq \varnothing$, then starting from any initial distribution $\mu\in\cP(\cS)$, every optimal policy under the liminf average-reward criterion achieves a long-run liminf average reward of $\alpha^*(\mu)$. In practice, even when a robust MDP is formulated under the long-run average criterion, good transient performance--where the system attains high expected stage-wise rewards while converging to its steady-state behavior--is highly desirable. In this section, we investigate this aspect of optimal robust policies by defining their transient values and constructing efficient robust optimal policies that enjoy (perhaps surprisingly) strong transient performance guarantees.

\subsection{Transient Values and Their Upper Bounds}
We begin by defining the transient value of a history-dependent policy $\pi\in\PiH(\cQ)$ against the optimal average reward $\alpha^*(\mu)$ defined in \eqref{eqn:alpha_star_mu}.

\begin{definition}[Transient Value]\label{def:TV}
Consider a weight function $w:\N\ra(0,1]$ that is non-increasing. Given an initial distribution $\mu\in\cP(\cS)$, $\pi\in\PiH(\cQ)$, we define the $w$-weighted transient value against $\alpha^*(\mu)$
\begin{equation}\label{eqn:weighted_TV}
  \mrm{TV}_w(\mu,\pi):= \inf_{p\in\cP} \liminf_{T\ra\infty}w(T)E_\mu^{\pi,p} \sum_{t=0}^{T-1}[r(X_t,A_t) -\alpha^*(\mu)].
\end{equation}
When $w(\cd)\equiv 1$, let $\mrm{TV}(\mu,\pi)$ denote the (unweighted) transient value. 
\end{definition}

The transient value captures the long-run cumulative deviation of the expected rewards under $(\mu,\pi)$ from the optimal average reward. First, we show in Proposition \ref{prop:TL_upper_bd} that history-dependent policies cannot attain an infinite unweighted transient value; that is, $\mrm{TV}(\mu,\pi) < \infty$. This, in turn, implies that $\mrm{TV}_w(\mu,\pi) \leq 0$ for any weighting function $w(T)\to 0$. However, in general, history-dependent optimal policies can exhibit arbitrarily poor transient performance.

Before we state Proposition \ref{prop:TL_upper_bd}, we introduce the following technical assumption. 
\begin{assumption}\label{assump:existence_of_pstar}
Given $\mu\in\cP(\cS)$, assume that there exists $p\in \cP$ such that $\alpha^*(\mu) = \sum_{s\in S}\mu(s)\alpha_{p}(s)$. In this case, we define $\cP^*_\mu:=\set{p\in\cP:\sum_{s\in S}\mu(s)\alpha_p(s) = \alpha^*(\mu)}$.  
\end{assumption}

We note that Assumption~\ref{assump:existence_of_pstar} serves as the minimal condition that makes the definition \eqref{eqn:weighted_TV} meaningful when $w(T)$ is not $O(1/T)$. Indeed, if there is no $p\in\cP$ such that $\alpha^*(\mu)=\sum_{s\in S}\mu(s)\alpha_p(s)$, then for any $\piRL\in\PiRL$,
$$
E_\mu^{\piRL,p}\sum_{t=0}^{T-1}[r(X_t,A_t)-\alpha^*(\mu)]
=
E_\mu^{\piRL,p}\sum_{t=0}^{T-1}[r(X_t,A_t)-\alpha_p(X_0)]
+
T\crbk{\sum_{s\in S}\mu(s)\alpha_p(s)-\alpha^*(\mu)}
=
\Omega(T)
$$
as $T\ra\infty$.

On the other hand, even if $\cP$ is compact and convex, a worst-case kernel need not exist \citep[Proposition~3.2]{grand2023beyond}. However, existence can be guaranteed when $\cP$ is weakly communicating. We summarize this fact in the following lemma as a side result. To streamline the discussion of transient values, we defer the proof of Lemma~\ref{lemma:exist_worst_case_kernel_WC} to Appendix~\ref{sec:proof:lemma:exist_worst_case_kernel_WC}.

\begin{lemma}[Existence of Worst-Case Kernels]\label{lemma:exist_worst_case_kernel_WC}
    Assume that $\cP$ is compact and weakly communicating. Then Assumption \ref{assump:existence_of_pstar} holds. 
\end{lemma}

With a well-defined finite transient value under Assumption \ref{assump:existence_of_pstar}, we establish some important properties.

\begin{proposition}[Properties of Transient Values]\label{prop:TL_upper_bd}
  Fix $\mu\in\cP(\cS)$ and suppose that Assumption \ref{assump:existence_of_pstar} is in force. Let 
  $$V^*:=\set{v\in\set{S\ra\R}:(\alpha_{p^*},v) \text{ solves \eqref{eqn:multichain_eqn} with $p = p^*$ for some } p^*\in\cP^*_\mu}.$$ Then, for any history-dependent policy $\pi$,
  \begin{equation}\label{eqn:TV_ub}\mrm{TV}(\mu,\pi) \leq \inf_{v\in V^*}\spnorm{v}.
  \end{equation}
  In particular, $\mrm{TV}_w(\mu,\pi)\leq 0$ for any weight function such that $\lim_{T\ra\infty}w(T) = 0$. 

  In contrast, the transient values of optimal policies are, in general, not lower bounded. Specifically, there exists an MDP instance that satisfies the following. For any weight function $w$ such that $Tw(T)\ra \infty$ as $T\ra\infty$, there exists a policy $\pi_w$ that is optimal in both the liminf and limsup average-reward criteria with $\mrm{TV}_w(\mu,\pi_w) = -\infty$. 
\end{proposition}

\begin{proof}[Proof of Proposition \ref{prop:TL_upper_bd}]
We first prove the upper bound \eqref{eqn:TV_ub}. Note that
$$\begin{aligned}
\mrm{TV}(\mu,\pi)&\leq \sup_{\pi\in\PiH(\cQ)}\inf_{p\in\cP}\liminf_{T\ra\infty} E_\mu^{\pi,p}\sum_{t=0}^{T-1}\sqbk{r(X_t,A_t) - \alpha^*(\mu)}\\
&\leq \inf_{p\in\cP}\sup_{\pi\in\PiH(\cQ)}\liminf_{T\ra\infty} E_\mu^{\pi,p}\sum_{t=0}^{T-1}\sqbk{r(X_t,A_t) - \alpha^*(\mu)}\\
&\leq\inf_{p^*\in\cP^*_\mu} \sup_{\pi\in\PiH(\cQ)}\liminf_{T\ra\infty} E_\mu^{\pi,p^*}\sum_{t=0}^{T-1}\sqbk{r(X_t,A_t) - \alpha^*(\mu)} 
\end{aligned}$$
By Assumption \ref{assump:existence_of_pstar}, we have that for all ${p^*\in\cP^*_\mu}$,  $\alpha^*(\mu)=\sum_{s\in S}\mu(s)\alpha_{p^*}(s)=E_\mu[\alpha_{p^*}(X_0)]$, where $X_0\sim \mu$.  Therefore, we further have 
$$
\mrm{TV}(\mu,\pi)  \leq \inf_{p^*\in\cP^*_\mu}\sup_{\pi\in\PiH(\cQ)}\liminf_{T\ra\infty} E_\mu^{\pi,p^*}\sum_{t=0}^{T-1}\sqbk{r(X_t,A_t) - \alpha_{p^*}(X_0)}.
$$

For any fixed $p^*\in\cP_\mu^*$, let $V_{p^*} = \set{v\in\set{S\ra\R}:(\alpha_{p^*},v) \text{ solves \eqref{eqn:multichain_eqn} with $p = p^*$}}$. Notice that $$\inf_{v\in V^*}\spnorm{v} = \inf_{p^*\in\cP^*_\mu} \inf_{v\in V_{p^*}}\spnorm{v}.$$
Therefore, to conclude \eqref{eqn:TV_ub}, it suffices to show that for any $\pi\in\PiH(\cQ)$, 
\begin{equation}\label{eqn:to_show_TV_ub}
\liminf_{T\ra\infty} E_\mu^{\pi,p^*}\sum_{t=0}^{T-1}\sqbk{r(X_t,A_t) - \alpha_{p^*}(X_0)}\leq \inf_{v\in V_{p^*}}\spnorm{v}.  
\end{equation}

To this end, we recall the multichain Bellman equation \eqref{eqn:multichain_eqn}. In particular, with $p$ replaced by $p^*$ and $(\alpha_{p^*},v_{p^*})$ as a solution pair, we have that 
$$\begin{aligned}\alpha_{p^*}(s)&\geq \sum_{s'\in S}p^*(s'|s,a)\alpha_{p^*}(s'),\\
v_{p^*}(s)&\geq r(s,a) - \alpha_{p^*}(s) + \sum_{s'\in S}p^*(s'|s,a)v_{p^*}(s')
\end{aligned}$$
for all $a\in A$. In particular, we have that for any $\pi\in\PiH(\cQ)$,
$$\begin{aligned}&E_{\mu}^{\pi,p^*}[\alpha_{p^*}(X_{t+1}) - \alpha_{p^*}(X_{t})|\cH_t]\leq 0,\\
&  E_{\mu}^{\pi,p^*}[r(X_t,A_t) - \alpha_{p^*}(X_{t}) +v_{p^*}(X_{t+1}) - v_{p^*}(X_t)|\cH_t]\leq 0. 
\end{aligned}$$

Summing over $t=0,\ds ,T-1$ and take expectation, we have that
\begin{equation}\label{eqn:to_use_ub_E_alpha_pstar} 0\geq E_{\mu}^{\pi,p^*}\sqbk{\sum_{t=0}^{T-1}\alpha_{p^*}(X_{t+1}) - \alpha_{p^*}(X_{t})} = E_{\mu}^{\pi,p^*}\alpha_{p^*}(X_{T}) -E_{\mu}^{\pi,p^*}\alpha_{p^*}(X_{0}), 
\end{equation}
and hence
$$\begin{aligned}
0&\geq E_{\mu}^{\pi,p^*}\sqbk{\sum_{t=0}^{T-1}r(X_t,A_t) - \alpha_{p^*}(X_{t}) +v_{p^*}(X_{t+1}) - v_{p^*}(X_t)}\\
& =E_{\mu}^{\pi,p^*}\sqbk{\sum_{t=0}^{T-1}r(X_t,A_t) - \alpha_{p^*}(X_{t}) }+E_{\mu}^{\pi,p^*}\sqbk{v_{p^*}(X_{T}) - v_{p^*}(X_0)}\\
&\geq E_{\mu}^{\pi,p^*}\sqbk{\sum_{t=0}^{T-1}(r(X_t,A_t) - \alpha_{p^*}(X_{0}) )}- \spnorm{v_{p^*}}
\end{aligned}$$
where the last inequality follows from \eqref{eqn:to_use_ub_E_alpha_pstar}, implying that $-E_{\mu}^{\pi,p^*}\alpha_{p^*}(X_{t}) \geq -E_{\mu}^{\pi,p^*}\alpha_{p^*}(X_{0})$ for all $t\geq 0$, together with $v_{p^*}(X_T) - v_{p^*}(X_0) \geq -\spnorm{v_{p^*}}$. Since $(\alpha_{p^*},v_{p^*})$ is an arbitrary solution pair to \eqref{eqn:multichain_eqn} and $\alpha_{p^*}$ is unique, 
\begin{equation}\label{eqn:to_use_TV_bd_using_pstar}E_{\mu}^{\pi,p^*}\sum_{t=0}^{T-1}\sqbk{r(X_t,A_t) - \alpha_{p^*}(X_{0}) } \leq \inf_{v\in V_{p^*}} \spnorm{v}.
\end{equation}
This implies \eqref{eqn:to_show_TV_ub} and hence \eqref{eqn:TV_ub}.

Next, we show that $\mrm{TV}_w(\mu,\pi)\leq 0$ for any weight function satisfying $\lim_{T\ra\infty}w(T) = 0$. 
The argument relies on the following elementary lemma, whose proof is deferred to Appendix \ref{sec:proof:lemma:liminf_of_weighted_seq}. 

\begin{lemma}\label{lemma:liminf_of_weighted_seq}
  Let $\set{g(T):{T\ge 1}}\subseteq   \mathbb{R}$ and $w:\N\ra (0,1]$ with $w(T)\to 0$ as $T\to\infty$. 
  Assume that $\liminf_{T\to\infty} g(T) \le C$ for some $C < \infty$. 
  Then, $\liminf_{T\to\infty} w(T) g(T)\le 0.$
\end{lemma}

To apply Lemma \ref{lemma:liminf_of_weighted_seq}, fix $p\in\cP$ and define 
$$g(\mu,\pi,p,T) := E_\mu^{\pi,p}\sum_{t=0}^{T-1}\sqbk{r(X_t,A_t) - \alpha^*(\mu)}.$$
By \eqref{eqn:to_use_TV_bd_using_pstar}, we have $\liminf_{T\ra\infty}g(\mu,\pi,p^*,T) < \infty$. 
Then, for any $w:\N\ra (0,1]$ with $\lim_{T\ra\infty}w(T) = 0$, applying Lemma \ref{lemma:liminf_of_weighted_seq} yields
$$\mrm{TV}_w(\mu,\pi) = \inf_{p\in\cP}\liminf_{T\to\infty} w(T) g(\mu,\pi,p,T)
\le \liminf_{T\to\infty} w(T) g(\mu,\pi,p^*,T) \leq 0.$$
This shows that the weighted transient value is always non-positive whenever the weighting function $w(T)$ vanishes as $T\to\infty$.

To show the final statement, we construct a simple robust MDP instance with $S = \set{s}$ and $A = \set{a_0,a_1}$. In this setting, there is only one possible transition kernel $p$, defined by $p(s|s,a_i) = 1$ for $i=0,1$. Therefore, the only possible ambiguity set is $\cP = \set{p}$. We let $\cQ = \cP(\cA)$ or $\set{\delta_{a_0},\delta_{a_1}}$, and consider a reward function $r(s,a_i) = r(a_i) = i$ for $i=0,1$. It is immediate that the optimal average reward is $1$, achieved by the policy that always plays action $a_1$.

For any fixed weight function $w$, we construct an optimal deterministic Markovian but non-stationary policy $\pi_w = (\pi_{w,0},\pi_{w,1},\pi_{w,2},\dots)$, where $\pi_{w,t}(a_i|s)\in\set{0,1}$, and show that $\mrm{TV}_w(\mu,\pi_w) = -\infty$.

To achieve this, consider a sequence of time points $\set{T_n:n\geq 1}\subseteq  \N$ satisfying $T_1=1$,
\begin{equation}\label{eqn:to_use_property_of_Tn}
T_{n+1}\ge 2T_n,\quad\text{and}\quad T_{n}w(T_n)\geq n^{2} + n.
\end{equation}
Such a sequence exists because $Tw(T)$ can be made arbitrarily large as $T$ increases. 
Next, define a sequence $\set{B_n\in\N:n\ge 1}$, to be specified later, satisfying $0\leq B_{n+1}-B_n\leq T_{n+1}-T_n$. Intuitively, $B_n$ represents the number of “bad” actions played before time $T_n$.

Specifically, we define the policy as follows. For $t = T_n,\dots,T_{n+1}-(B_{n+1}-B_n)-1$, the policy plays the good action, i.e., $\pi_{w,t}(a_1|s) = 1$. For the remaining time steps $t = T_{n+1}-(B_{n+1}-B_n),\dots,T_{n+1}-1$, the policy plays the bad action, i.e., $\pi_{w,t}(a_0|s) = 1$.

For $T_n \leq T \leq T_{n+1}-1$, since $\pi_w$ first plays the good action and then the bad action, the fraction of bad actions played satisfies
$$
0\leq 1- \frac{1}{T}E_\mu^{\pi,p}\sum_{t=0}^{T-1}r(X_t,A_t) \leq \max\set{\frac{B_{n}}{T_{n}},\frac{B_{n+1}}{T_{n+1}}}.
$$
Therefore, if $T_n^{-1}B_n\to 0$, then $\pi_w$ is an optimal policy under both the limsup and liminf average-reward criteria. On the other hand,
$$
w(T_n)E_\mu^{\pi,p}\sum_{t=0}^{T_n-1}[r(X_t,A_t)-1] = -w(T_n)B_n,
$$
so $\mrm{TV}_w(\mu,\pi_w) = -\infty$ if $\pi_w$ is optimal and $w(T_n)B_n\to\infty$.

Hence, to finish the proof, it suffices to find $B_n$ satisfying $B_{n+1}-B_n \leq T_{n+1}-T_n$, $T_n^{-1}B_n\to 0$, and $w(T_n)B_n\to\infty$. 
Let $B_n := \lfloor T_n/n \rfloor$. Then,
$$
B_{n+1} - B_n \leq \floor{\frac{T_{n+1}}{n+1}} \leq \floor{\frac{T_{n+1}}{2}}  \leq T_{n+1}-T_n,
$$
where the last inequality follows from \eqref{eqn:to_use_property_of_Tn}. Moreover, by the same conditions,
$$
\frac{B_n}{T_n}\leq \frac{1}{n}\to 0,
\quad\text{and}\quad
w(T_n)B_n \geq \frac{T_nw(T_n)}{n}-1 \geq n\to\infty.
$$
This completes the proof.
\end{proof}

From Proposition \ref{prop:TL_upper_bd}, if $\mrm{TV}_w(\mu,\pi) > -\infty$, then $1/w(T)$ can be interpreted as the growth rate of the deviation between the cumulative reward up to time $T$ and $T$ times the optimal average reward. In particular, when $w(T) \geq \Omega(1/T)$ (for example, $w(T) = 1/\sqrt{T}$), the condition $\mrm{TV}_w(\mu,\pi) > -\infty$ provides a stronger form of optimality certificate than the long-run average criterion, formally suggesting that
$$\inf_{p\in\cP}\frac{1}{T} E_\mu^{\pi,p}\sum_{t=0}^{T-1}r(X_t,A_t) = \alpha^*(\mu) - O\crbk{\cfrac{1}{Tw(T)}}.$$

On the other hand, when $w(T) = \frac{1}{T}$, we have
$$\begin{aligned} 
\mrm{TV}_w(\mu,\pi) &= \inf_{p\in\cP}\liminf_{T\ra\infty} \sqbk{\frac{1}{T}E_\mu^{\pi,p}\sum_{t=0}^{T-1}r(X_t,A_t) }- \alpha^*(\mu)\\
&=\inf_{p\in\cP}\underline{\alpha}(\mu,\pi,p) -\alpha^*(\mu).
\end{aligned}$$ Hence, when $\PiRL(\cQ,\cP) \neq \varnothing$, by Theorem \ref{thm:value_RL_policy}, any $\piRL\in \PiRL(\cQ,\cP)$ satisfies $\mrm{TV}_w(\mu,\piRL) = 0$ for all $\mu\in\cP(\cS)$. 
In contrast, any suboptimal policy $\pi$ under the liminf average-reward criterion satisfies $-1 \leq \mrm{TV}_w(\mu,\pi) < 0$. 
Therefore, suboptimal policies have negative transient values, with magnitude growing at rate $O(1/w(T)) = O(T)$.

\subsection{Transient Value Lower Bounds for RL Policies}

We have shown that finite lower bounds on $\mrm{TV}_w$ translate into transient performance guarantees for an optimal policy. We now derive such lower bounds for RL policies $\piRL\in\PiRL(\cQ,\cP)$ by leveraging expected regret upper bounds. Recall that, by Definition \ref{def:online_RL}, any $\piRL\in\PiRL(\cQ,\cP)$ satisfies $E_{\mu}^{\piRL,p}R_{r,p}(T) = o(T)$. The proposition below quantifies how sublinear regret rates yield a corresponding lower bound on $\mrm{TV}_w$.

\begin{proposition}[Transient Value Lower Bound for RL Policies]\label{prop:TV_lower_bd_RL_policy}
Fix $\mu\in\cP(\cS)$ and $\piRL\in\PiRL(\cQ,\cP)$. Suppose Assumption \ref{assump:existence_of_pstar} holds. Moreover, assume that for any $p\in\cP$, there exists a constant $c_{p} > 0$ and a weight function $w_p$ such that $Tw_p(T)\ra\infty$ and 
\begin{equation}\label{eqn:regret_bd_RL_policy}
  E_{\mu}^{\piRL,p}R_{r,p}(T) \leq \frac{c_{p}}{w_p(T)}
\end{equation}
for all $T\geq 1$. Then, for all weight functions $w$ such that $Tw(T)\ra\infty$, 
$$\mrm{TV}_{w}(\mu,\piRL) \geq -\sup_{p^*\in\cP^*_\mu}\limsup_{T\ra\infty}\frac{c_{p^*}w(T)}{w_{p^*}(T)}.$$
\end{proposition}

\begin{remark}\label{rmk:TV_growth_rate}
Before presenting the proof, we note that Proposition \ref{prop:TV_lower_bd_RL_policy} identifies how the decay rate of $w(T)$ needed to ensure $\mrm{TV}_{w}(\mu,\piRL) > -\infty$ is governed by the regret growth rate of $\piRL$ against the \textit{worst-case kernels} $p^*\in\cP^*_\mu$ that attain the robust optimal average reward $\alpha^*(\mu)$. Importantly, this dependence is through the quantities $w_{p^*}(T)$, and therefore can be less stringent than what would be suggested by the worst-case rate $\sup_{p\in\cP} w_p(T)$.

In particular, a typical minimax-optimal RL policy achieves $\widetilde{O}(\sqrt{T})$ \citep{jaksch2010near,zhang2023amdp_regret} regret for all weakly communicating $p\in\cP$, with a common growth rate but instance-dependent constants. In this case, one may take $w_p(T)=\widetilde{\Omega}(T^{-1/2})$ for all such $p$. For instance, applying Proposition \ref{prop:TV_lower_bd_RL_policy} to the RL policy $\pi_{\mathrm{ZX}}$ in Example \ref{example:RL_alg_ZX}, we obtain that under the assumptions of Corollary \ref{cor:piZX_RL}, for any weight function $w(T)=O(T^{-1/2-\epsilon})$, we have $\mrm{TV}_w(\mu,\pi_{\mathrm{ZX}})=0$ for all $\mu\in\cP(\cS)$.

On the other hand, if a policy $\pi$ is designed to adapt to the sub-optimal $p\in\cP$ and, in addition, achieves a \emph{slower} regret growth rate (ideally $O(1)$) against the particular kernels $p^*\in\cP^*_\mu$, then Proposition \ref{prop:TV_lower_bd_RL_policy} permits $w(T)$ to decay more slowly while still ensuring $\mrm{TV}_{w}(\mu,\pi) > -\infty$. This highlights that transient values of optimal policies in robust MDPs can behave fundamentally differently from minimax regret in standard online RL. We leverage this observation in the next section to construct robust optimal policies with constant-order transient values.
\end{remark}

\begin{proof}[Proof of Proposition \ref{prop:TV_lower_bd_RL_policy}]
We note that for any $p\in\cP$, by the definition of the regret in \eqref{eqn:regret_def},
$$\begin{aligned}
E_\mu^{\piRL,p}\sum_{t=0}^{T-1}[r(X_t,A_t) - \alpha^*(\mu)] &= E_\mu^{\piRL,p}\sqbk{\sum_{t=0}^{T-1}r(X_t,A_t) }- T\alpha^*(\mu)\\
&= -E_\mu^{\piRL,p}R_{r,p}(T) + T\sqbk{\sum_{s\in S}\mu(s)\alpha_p(s) - \alpha^*(\mu)}.
\end{aligned}$$
Therefore, by \eqref{eqn:regret_bd_RL_policy}, we have that for any $p\in\cP$,
$$\begin{aligned}  
E_\mu^{\piRL,p}\sum_{t=0}^{T-1}[r(X_t,A_t) - \alpha^*(\mu)] &\geq - \frac{c_{p}}{w_p(T)} + T\sqbk{\sum_{s\in S}\mu(s)\alpha_p(s) - \alpha^*(\mu)}.
\end{aligned}$$
In particular, if $p\notin\cP^*_\mu$, then $\sum_{s\in S}\mu(s)\alpha_p(s) - \alpha^*(\mu) > 0$, and hence for any weight function $Tw(T)\ra\infty$,
$$\liminf_{T\ra\infty} w(T)E_\mu^{\piRL,p}\sum_{t=0}^{T-1}[r(X_t,A_t) - \alpha^*(\mu)] = \infty,$$ 
while for $p^*\in\cP^*_\mu$, 
$$\begin{aligned}\liminf_{T\ra\infty} w(T)E_\mu^{\piRL,p^*}\sum_{t=0}^{T-1}[r(X_t,A_t) - \alpha^*(\mu)] \geq\liminf_{T\ra\infty}-\frac{c_{p^*}w(T)}{w_{p^*}(T)}.
\end{aligned}$$

Therefore, by the definition of the transient value,
$$\begin{aligned}
\mrm{TV}_w(\mu,\piRL) &= \inf_{p\in\cP}\liminf_{T\ra\infty} w(T)E_\mu^{\piRL,p}\sum_{t=0}^{T-1}[r(X_t,A_t) - \alpha^*(\mu)]\\
&\stackrel{(i)}{=} \inf_{p\in\cP^*_\mu}\liminf_{T\ra\infty} w(T)E_\mu^{\piRL,p}\sum_{t=0}^{T-1}[r(X_t,A_t) - \alpha^*(\mu)]\\
&\geq -\sup_{p^*\in\cP^*_\mu}\limsup_{T\ra\infty}\frac{c_{p^*}w(T)}{w_{p^*}(T)}.
\end{aligned}$$
where $(i)$ follows from the fact that for any $p\notin\cP^*_\mu$, the corresponding limit is positive infinity. 
\end{proof}

\section{A Policy with Constant Order Transient Value}
This section develops and analyzes an optimal policy with constant-order transient value ($w \equiv 1$) of the same order as the bias span of an average-reward optimal model, rather than diverging with time.

Assuming a weakly communicating scenario, the design of this policy draws inspiration from the observation in Remark \ref{rmk:TV_growth_rate}, as well as the fact that the controller ``knows'' the optimal adversarial response, i.e., $p^*$ such that $\alpha_{p^*}=\alpha^*$, so long as its policy achieves online RL.

Specifically, we construct a policy that combines (i) an optimal stationary policy $\Delta^*$ for the worst-case kernel $p^*$ with (ii) an anytime-valid sequential test that detects when the observed trajectory is inconsistent with the induced optimal Markov kernel $p^*_{\Delta^*}$ and (iii) a reference online RL policy $\piRL\in\PiRL(\cQ,\cP)$. 

This policy starts by drawing decisions from $\Delta^*$ while computing the test statistics along the way. If the test never rejects and we keep using $\Delta^*$, then the transition dynamics behave essentially like the Markov chain under $(\Delta^*,p^*)$. By the characterization \eqref{eqn:wc_Bellman_eqn}, this yields the optimal average reward $\alpha^*$ with $-O(\spnorm{v_{p^*}})$ transient value. On the other hand, if the test rightfully rejects, then the policy switches to a reference online RL policy $\pi_{\mathrm{RL}}$, ensuring that under any strictly suboptimal kernel the transient value grows without bound.

Central to the design of this policy is the calibration of the sequential test. We need to ensure that false rejections, i.e., type-I error, occur with sufficiently small probability $\rho$, while true rejections occur quickly, with expected delay $O(\log(1/\rho))$ whenever the alternative differs from $p^*_{\Delta^*}$; see Theorem~\ref{thm:rej_time}. These ingredients are then assembled into the epoch-based policy $\pi^*$ and analyzed to yield an explicit constant-order lower bound on $\mrm{TV}(\mu,\pi^*)$; see Theorem~\ref{thm:pistar_tv_bound}.

\subsection{A Brief Introduction to Sequential Probability Ratio Tests for Markov Chains}

A key component of our construction is a sequential likelihood-ratio-type statistic that continuously tests the null hypothesis that the trajectory is generated by the kernel $P_0 := p^*_{\Delta^*}$ against a composite alternative in which the trajectory is generated by some other kernel. This is in the spirit of Wald’s sequential probability ratio test \citep{wald1945sprt}, adapted to Markovian dependence and to a composite alternative via a mixture likelihood ratio.

We proceed with defining the mixture likelihood ratio process and the rejection time. 

\begin{definition}[Mixture Likelihood Ratio Process and Rejection Time]\label{def:mixture_RL_process_and_rej_time}
Let $P_0$ be the null transition kernel and let $\Pi\in\cP(\cB(\R^{|S|\times|S|}))$ be a prior on transition kernels. For $n\ge 1$, define the mixture likelihood ratio process
\begin{equation}\label{eqn:def_mixture_LR}
\Lambda_n
:=
\int \prod_{t=0}^{n-1}\frac{Q(X_{t+1}| X_t)}{P_0(X_{t+1}| X_t)} \Pi(dQ),
\end{equation}
with the convention $q/0=+\infty$ for $q>0$.
Define for $\rho\in(0,1)$ the rejection time
\begin{equation}\label{eqn:rej_time}
\tau_\rho:=\inf\set{n\ge 1: P_0(X_n| X_{n-1})=0\ \text{or}\ \Lambda_{n}\ge \frac1\rho}.
\end{equation}

\end{definition}

\begin{definition}[Product Dirichlet Prior]\label{def:dirichlet-prior}
Fix parameters $\set{\gamma(s' | s) > 0:s,s'\in S}$.
Let $\Pi$ be the \emph{product Dirichlet prior}, defined as the law of random matrix $\Phi\in\R^{|S|\times |S|}$ with independent rows, where for each $s\in S$, $\Phi(\cd | s)\sim \mrm{Dirichlet}\crbk{\gamma(\cd | s)}$
and $\Phi(\cd | s)$ is independent of $\Phi(\cd | s')$ for all $s\neq s'$.

\end{definition}

Before we state the main theorem of this section, we define the following notation and terminology. For a transition kernel $Q$, viewed as a matrix in $\R^{S\times S}$, let $Q|_C$ denote the sub-kernel on $C\subseteq   S$, i.e., the principal submatrix of $Q$ indexed by $(s,s')\in C\times C$. Also, let $\cC_Q$ denote the set of closed communicating classes of $Q$. We say that a transition kernel $Q$ is unichain if $\cC_Q = \set{C_0}$ for some $C_0\subseteq   S$; i.e., $Q$ has only one closed communicating class, possibly with some transient states.

\begin{theorem}[Properties of the Rejection Time]\label{thm:rej_time}
Let $\Pi$ be the product Dirichlet prior with parameter $\gamma$ in Definition \ref{def:dirichlet-prior}. Let $P_0$ be unichain with $C_0\subseteq   S$ as its only closed communicating class. Then the following properties hold for any $\mu\in\cP(\cS)$: 
\begin{enumerate}
    \item The type-I-error of ever rejecting under $P_0$ is bounded by $\rho$; i.e., 
    \begin{equation}\label{eqn:type_1_error}
        P_\mu^{P_0}(\tau_\rho < \infty)\leq \rho
    \end{equation}
    \item  Let $P$ be an alternative transition kernel. Suppose $P|_{C_0}\neq P_0|_{C_0}$, then there exists some constant $K <\infty$ that only depend on $(P,P_0,\gamma)$ such that 
    \begin{equation}\label{eqn:E_rej_time_ub}
        E_\mu^P\tau_\rho \leq K\crbk{\log\frac 1 \rho +1}
    \end{equation}
\end{enumerate}
\end{theorem}
The proof of this result is presented in Section \ref{sec:proof:thm:rej_time}. We note that the expected rejection time scales with $\log(1/\rho)$ instead of $1/\rho$, an important observation that enables the design of our policy in Section \ref{sec:construct_policy}. 
\subsubsection{Proof of Theorem \ref{thm:rej_time}}\label{sec:proof:thm:rej_time}

\begin{proof} For clarity, we prove the two statements of Theorem \ref{thm:rej_time} separately, each via a stand-alone proposition. The proofs of these propositions are presented in Section \ref{sec:proof:propositions_for_thm_rej_time}.

We show a more general version of \eqref{eqn:type_1_error} summarized in the following Proposition \ref{prop:type-I-error}.
\begin{proposition}[Type-I error control]\label{prop:type-I-error}
Let $P_0$ be any transition kernel on $S$ and let $\Pi$ be any prior. For every $\rho\in(0,1)$, the type-I error of ever rejecting under $P_0$ is bounded by $\rho$; i.e., $P_\mu^{P_0}\crbk{\tau_\rho < \infty}\leq \rho.$
\end{proposition}
This proposition immediately implies the first statement of Theorem \ref{thm:rej_time}.

Next, we prove the second statement of Theorem \ref{thm:rej_time}. We first present Lemma \ref{lemma:dirichlet-prior}, which states an important property of the product Dirichlet prior in Definition \ref{def:dirichlet-prior}. This property underlies our analysis by turning the mixture likelihood ratio process $\Lambda_n$ into a likelihood ratio process for testing against a simple alternative $P$. The proof of Lemma \ref{lemma:dirichlet-prior} is deferred to \ref{sec:proof:lemma:dirichlet-prior}.

\begin{lemma}[Uniformity of Product Dirichlet Prior]\label{lemma:dirichlet-prior}

Let $\Pi$ be the product Dirichlet prior with parameter $\gamma$ in Definition \ref{def:dirichlet-prior}. Then, for any transition kernel $P$ and every $\epsilon>0$, the set
\begin{equation}\label{eqn:def-Ueps}
U(\epsilon,P):=\set{Q:\ Q(s' | s)\ge e^{-\epsilon}P(s' | s)\ \ \forall s,s'\in S}
\end{equation}
satisfies $\Pi(U(\epsilon,P))>0$.
\end{lemma}

We define the following parameter
\begin{equation}\label{eqn:def_beta}
\beta(\rho,\epsilon):=\log\frac{1}{\rho \Pi(U(\epsilon,P))}.   
\end{equation}
By Lemma \ref{lemma:dirichlet-prior}, $\beta(\rho,\epsilon) < \infty$ for any $\epsilon > 0$ and $\rho\in(0,1)$. 

Let $P$ be a transition kernel on a finite state space $S$ and $C\subseteq   S$ be one of its closed communicating classes. We denote $\nu_{P,C}$ as the unique stationary distribution of a Markov chain with transition kernel $P$ restricted to a closed communicating class $C$. Moreover, for any $C\in\cC_P$, we define the following $\nu_{P,C}$-weighted KL-divergence. 
\begin{definition}[KL-Divergence between Transition Kernels]
  Consider two transition kernels $P$ and $P_0$ on $S$. Let $C$ be a closed communicating class of $P$. We define the stationary distribution weighted KL-divergence as 
  \begin{equation}\begin{aligned}I_C(P\|P_0) &:= \sum_{s\in C}\nu_{P,C}(s)\sum_{s'\in S}P(s'|s)\log\crbk{\frac{P(s'|s)}{P_0(s'|s)}} \\
  &= \sum_{s\in C}\nu_{P,C}(s)\mrm{KL}(P(\cd|s)\|P_0(\cd|s)).
  \end{aligned}
  \end{equation}
  where $\mrm{KL}(\cd\|\cd)$ is the KL divergence, with $\mrm{KL}(\mu\|\nu)=\infty$ if $\mu(s)>0$ and $\nu(s)=0$ for some $s$.
\end{definition}

Then, the second statement of Theorem \ref{thm:rej_time} will follow from the following proposition, whose proof is deferred to the next section. 

\begin{proposition}[Bounds on Expected Rejection Time]\label{prop:E_rej_time}
Fix an alternative transition kernel $P$ on $S$ and $C\in\cC_P$, let $T_C:=\inf\set{t\ge 0:X_t\in C}$ and write $\1_C:=\1\set{T_C<\infty}.$

If $I_C(P\|P_0)=\infty$, then there exists $K_C<\infty$, depending only on $( P,P_0,C)$, such that for all $\rho\in(0,1)$,
\begin{equation}\label{eqn:Etau-inf}
E_\mu^P\tau_\rho \1_C
\le K_C.
\end{equation}

If $0<I_C(P\|P_0)<\infty$, let $\epsilon:=\frac12 I_C(P\|P_0)$ and let $\beta(\rho,\epsilon)$ be as in \eqref{eqn:def_beta}.
There exists $K_C<\infty$, depending only on $(P,P_0,C)$, such that for all $\rho\in(0,1)$,
\begin{equation}\label{eqn:Etau-fin}
E_\mu^P\tau_\rho \1_C
\le K_C(1+\beta(\rho,\epsilon)).
\end{equation}

\end{proposition}

To apply Proposition \ref{prop:E_rej_time}, we verify that if $P|_{C_0}\neq P_0|_{C_0}$, $I_C(P\|P_0) > 0$ for all $C\in\cC_P$. We show the contrapositive statement that, if $I_C(P\|P_0) = 0$ for some $C\in\cC_P$, then $P|_{C_0}= P_0|_{C_0}$. 

Since $C$ is a closed communicating class of $P$, $\nu_{P,C}(s)>0$ for all $s\in C$. So, if $I_C(P\|P_0)=0$, $$\mathrm{KL}(P(\cdot | s)\|P_0(\cdot | s)) = 0$$ for every $s\in C$. Hence $P(\cdot| s)=P_0(\cdot| s)$ for all $s\in C$.

Because $C$ is closed under $P$, for each $s\in C$ we have $P(C | s)=1$, hence also $P_0(C | s)=1$. Thus $C$ is closed under $P_0$. Since $P_0$ has a unique closed communicating class $C_0$, it follows that $C\supseteq C_0$. The equality $P(\cdot | s)=P_0(\cdot | s)$ for all $s\in C$ implies $P|_{C_0}=P_0|_{C_0}$.

Therefore, under the assumptions of the second statement of Theorem \ref{thm:rej_time}, $I_C(P\|P_0) > 0$ for all $C\in\cC_P$.  Hence, Proposition \ref{prop:E_rej_time} applies, and with
\[
K := \crbk{\sum_{C\in\cC_P} K_C }\cd\crbk{ 1+\max_{C\in\cC_P,\ I_C(P\|P_0) < \infty}\log\sqbk{\frac{1}{\Pi\crbk{U(\frac12 I_C(P\|P_0), P))}}}}
\]
depending only on $(P,P_0,\gamma)$, we have
\[
E_\mu^P\tau_\rho
=\sum_{C\in\cC_P}E_\mu^P\tau_\rho \1\set{T_C<\infty}
\leq K\crbk{\log\frac 1 \rho +1}.
\]
This completes the proof of Theorem \ref{thm:rej_time}.

\end{proof}

\subsubsection{Proof of Propositions \ref{prop:type-I-error} and \ref{prop:E_rej_time}}\label{sec:proof:propositions_for_thm_rej_time}

\begin{proof}[Proof of Proposition \ref{prop:type-I-error}]
We first note that
\begin{align*}
P_\mu^{P_0}\crbk{\set{\exists\, n\ge 1:\ P_0(X_n | X_{n-1})=0}}=0.
\end{align*}
Therefore,
\begin{equation}\label{eqn:to_use_type-I-error_ub}
P_\mu^{P_0}\crbk{\tau_\rho<\infty}
= P_\mu^{P_0}\crbk{\sup_{n\ge 0}\Lambda_n\ge \frac1\rho}.
\end{equation}

To analyze the r.h.s. probability, we introduce Ville's inequality for nonnegative supermartingales \citep{ville1939}. For completeness, we include a proof in Appendix \ref{sec:proof:lemma:ville_bound}. 

\begin{lemma}[Ville's Inequality]\label{lemma:ville_bound}
Let $\sigma(Z_t:t\leq n)\subseteq  \cF_n$. If $\set{Z_n:n\ge 0}$ be a nonnegative $(P,\cF_n)$-supermartingale with $Z_0=1$, then for every $c >0$,
\begin{equation}\label{eqn:ville_bound}
P\crbk{\sup_{n\ge 0} Z_n \ge \frac{1}{c}}\le c .
\end{equation}
\end{lemma}

Note that $\Lambda_0=1$ by definition and $\sigma(\Lambda_t:t\leq n)\subseteq   \cF_n$. To apply Lemma \ref{lemma:ville_bound}, we show that $\Lambda_n$ is a nonnegative supermartingale. 

Fix a transition kernel $Q$ and define
\[
G_n(Q)=\prod_{t=0}^{n-1}\frac{Q(X_{t+1} | X_t)}{P_0(X_{t+1} | X_t)}
\] with the convention that $q/0 = +\infty$ for $q>0$. 
Note that $\Lambda_n=\int G_n(Q)d\Pi(Q)$ and 
\[
G_{n+1}(Q)
=G_n(Q) \frac{Q(X_{n+1} | X_n)}{P_0(X_{n+1} | X_n)}.
\]

We consider the following conditional expectation
\begin{align*}
E_\mu^{P_0}\sqbkcond{\frac{Q(X_{n+1} | X_n)}{P_0(X_{n+1} | X_n)}}{\cF_n}
&=E_{X_n}^{P_0}\sqbk{\frac{Q(X_{1} | X_0)}{P_0(X_{1} | X_0)}}\\
&= \sum_{s\in S, P_0(s|X_n) > 0}P_0(s|X_n)\frac{Q(s|X_n)}{P_0(s|X_n)}\\
&= \sum_{s\in S,P_0(s|X_n) >0 }Q(s|X_n) \\
&\leq 1
\end{align*}
and hence $E_\mu^{P_0}\sqbk{G_{n+1}(Q) | \cF_n}\leq G_n(Q)$. By Tonelli's theorem,
\begin{align*}
E_\mu^{P_0}\sqbk{\Lambda_{n+1} | \cF_n}
&=
E_\mu^{P_0}\sqbkcond{\int G_{n+1}(Q)\Pi(dQ)}{\cF_n} \\
&=
\int E_\mu^{P_0}\sqbk{G_{n+1}(Q) | \cF_n}\Pi(dQ) \\
&\leq\int G_n(Q)\Pi(dQ) \\
&=
\Lambda_n;
\end{align*}
i.e., $\set{\Lambda_n:n\ge 0}$ is a nonnegative supermartingale under $P_\mu^{P_0}$.

Therefore, applying Lemma \ref{lemma:ville_bound} yields
\[
P_\mu^{P_0}\crbk{\sup_{n\ge 0}\Lambda_n\ge \frac1\rho}\le \rho.
\]
Combining this with \eqref{eqn:to_use_type-I-error_ub} proves Proposition \ref{prop:type-I-error}.
\end{proof}

\begin{proof}[Proof of Proposition \ref{prop:E_rej_time}]To prove Proposition \ref{prop:E_rej_time}, we first turn the mixture likelihood ratio process $\Lambda_n$ into a likelihood ratio process for testing against a simple alternative $P$. To this end, we consider for $Q\in U(\epsilon,P)$
\begin{equation}\label{eqn:Ueps-dominates}
\prod_{t=0}^{n-1}\frac{Q(X_{t+1}| X_t)}{P_0(X_{t+1}| X_t)}
\ge e^{-\epsilon n}\prod_{t=0}^{n-1}\frac{P(X_{t+1}| X_t)}{P_0(X_{t+1}| X_t)}
\end{equation}
for all $n\ge 1$. With this observation and using the product Dirichlet prior $\Pi$ as in Definition \ref{def:dirichlet-prior}, we apply Lemma~\ref{lemma:dirichlet-prior} to get the following lemma, whose proof is deferred to Appendix \ref{sec:proof:lemma:mixture-reduction}. 
\begin{lemma}[Mixture-to-fixed-$P$ Reduction]\label{lemma:mixture-reduction}
Define
\[
S_n(\epsilon):=\sum_{t=0}^{n-1}\crbk{\log \frac{P(X_{t+1}| X_t)}{P_0(X_{t+1}| X_t)}-\epsilon },
\]
and $\sigma(\epsilon):=\inf\set{n\ge 1:S_n(\epsilon)\ge \beta(\rho,\epsilon)}.$
Then $\tau_\rho\le \sigma(\epsilon)$.
\end{lemma}

Lemma \ref{lemma:mixture-reduction} reduces the problem of bounding the hitting time of the mixture likelihood ratio process \eqref{eqn:def_mixture_LR} to bounding the hitting time of a simple likelihood ratio process that tests $P_0$ against $P$. Moreover, the expectation of the hitting time is taken with respect to a Markov chain with transition kernel $P$. In particular, if $P$ and $P_0$ are irreducible, then by the Markov chain law of large numbers, $n\inv S_n(\epsilon)\ra I_S(P\|P_0) - \epsilon$ a.s.$P_\mu^P$. Hence, $\sigma(\epsilon) < \infty$ a.s.$P_\mu^P$ if $I_S(P\|P_0) - \epsilon > 0$.

With this intuition, we consider the general case where $P$ may have multiple recurrent classes.

\noindent\textbf{Case 1: $I_C(P\|P_0)=\infty$.}

Since $P$ restricted to $C$ is irreducible and $\nu_{P,C}$ has full support on $C$, the condition $I_C(P\|P_0)=\infty$ implies that there exist $s_*,s_*'\in C$ such that
$P(s_*'| s_*)>0$ but $P_0(s_*'| s_*)=0$.
Define the edge-hitting time
\[
T_*:=\inf\set{t\ge 1:X_{t-1}=s_*, X_t=s_*'}.
\]
By definition of $\tau_\rho$, whenever $P_0(X_t| X_{t-1})=0$ we have that $P(X_t|X_{t-1}) > 0$ a.s.$P_\mu^P$; hence, $\tau_\rho\le t$. Therefore, $\tau_\rho\leq T_*$ a.s.$P_\mu^P$. 

Apply the strong Markov property at $T_C$, we have 
$$E_\mu^P\tau_\rho \1_C\leq E_\mu^PT_* \1_C
\le
E_\mu^PT_C \1_C
+\max_{s\in C}E_s^P T_* \cdot P_\mu^P(T_C<\infty).$$
It remains to note that $\max_{s\in C}E_s^P T_*<\infty$: on the finite irreducible chain on $C$, every state has finite expected hitting time to $s_*$, and each visit to $s_*$ has success probability $P(s_*'| s_*)>0$ to traverse the edge $(s_*,s_*')$ on the next step, so the number of visits to $s_*$ before success has geometric tails, yielding uniform finiteness. This gives \eqref{eqn:Etau-inf}.

\noindent\textbf{Case 2: $0<I_C(P\|P_0)<\infty$.}

Set $\epsilon:=\frac12 I_C(P\|P_0)$, and define $\sigma(\epsilon)$ as in Lemma~\ref{lemma:mixture-reduction}.
Lemma~\ref{lemma:mixture-reduction} gives
\begin{equation}\label{eqn:tau-le-sigma}
\tau_\rho\le \sigma(\epsilon)\qquad\text{a.s.\ }P_\mu^P.
\end{equation}
Fix a reference state $s_C\in C$ and define the successive return times
\[
\theta_0:=\inf\set{t\ge 0:X_t=s_C},
\qquad
\theta_{k+1}:=\inf\set{t>\theta_k:X_t=s_C}.
\]
On $\set{T_C<\infty}$, the chain enters $C$ and, since $P$ is irreducible on $C$, hits $s_C$ in finite time; hence $\theta_0<\infty$ a.s.\ on $\1_C=1$.
Define the cycle lengths and cycle log-likelihood increments (for $k\ge 1$)
\[
\xi_k:=\theta_k-\theta_{k-1},
\qquad
M_k:=\sum_{t=\theta_{k-1}}^{\theta_k-1}\crbk{\log\frac{P(X_{t+1}| X_t)}{P_0(X_{t+1}| X_t)}-\epsilon}.
\]
Note that $X_{\theta_{k}} = s_C$, so $M_k\in \sigma(X_{\theta_{k-1}},\ds ,X_{\theta_{k}-1})$. Therefore, by the regeneration property of Markov chains at the return times $\theta_k$, $\set{(\xi_k,M_k):k\ge 1}$ is i.i.d. under $P_{s_C}^P$.

Let $\beta:=\beta(\rho,\epsilon)$ and define $S_n(\epsilon)=\log L_n(P,P_0)-\epsilon n$ as in Lemma~\ref{lemma:mixture-reduction}.
For $k\ge 1$,
\[
S_{\theta_k}(\epsilon)=S_{\theta_0}(\epsilon)+\sum_{i=1}^k M_i.
\]
Define, for $y\in\R$, the cycle counter
\[
N(y):=\inf\set{k\ge 0:y+\sum_{i=1}^k M_i\ge \beta}
\]
and the associated expected additional time
\[
g(y):=E_{s_C}^P\sqbk{\sum_{i=1}^{N(y)}\xi_i}.
\]
Then on $\1_C=1$, $\sigma(\epsilon)\le \theta_{N(S_{\theta_0}(\epsilon))}$, hence by the strong Markov property at $\theta_0$,
\begin{equation}\label{eqn:sigma-reg-bound}
E_\mu^P\sigma(\epsilon) \1_C
\le
E_\mu^P\theta_0 \1_C
+
E_\mu^Pg\crbk{S_{\theta_0}(\epsilon)} \1_C.
\end{equation}

To use \eqref{eqn:sigma-reg-bound}, we first bound $g(y)$ with the following lemma, whose proof is deferred to Appendix \ref{sec:proof:lemma:regen-crossing}. 
\begin{lemma}\label{lemma:regen-crossing}
There exists $\kappa_C<\infty$, depending only on $(P,P_0,C)$, such that
\begin{equation}\label{eqn:regen-crossing-bound}
g(y)\le \frac{4}{I_C(P\|P_0)} (\beta-y)_+ + \kappa_C,
\qquad \forall y\in\R.
\end{equation}
\end{lemma}

Applying Lemma~\ref{lemma:regen-crossing} with $y=S_{\theta_0}(\epsilon)$ yields
\begin{equation}\label{eqn:gS-bound-0}
E_\mu^P\sqbk{g\crbk{S_{\theta_0}(\epsilon)} \1_C}
\le
\frac{4}{I_C(P\|P_0)} E_\mu^P\sqbk{(\beta-S_{\theta_0}(\epsilon))_+ \1_C}
+\kappa_C P_\mu^P(T_C<\infty).
\end{equation}

Next, define
\[
l_{P,P_0}:=\min\set{\log\frac{P(s' | s)}{P_0(s' | s)}: P(s' | s)>0,\ P_0(s' | s)>0}\in\R.
\]
Then, every realized transition has positive probability; i.e., $P(X_{t+1} | X_t)>0$ a.s.$P_\mu^P$. If also $P_0(X_{t+1} | X_t)>0$ then the log-ratio is bounded below by $l_{P,P_0}$ by definition, while if $P_0(X_{t+1} | X_t)=0$ the log-ratio is $+\infty$ by convention. Hence,
\[
\log\frac{P(X_{t+1} | X_t)}{P_0(X_{t+1} | X_t)}\ge l_{P,P_0},
\]
a.s.$P_\mu^P$ and therefore
\begin{equation}\label{eqn:S-lower}
S_{\theta_0}(\epsilon)\ge (l_{P,P_0}-\epsilon)\theta_0
\implies
(\beta-S_{\theta_0}(\epsilon))_+\le \beta + (\epsilon-l_{P,P_0})_+\theta_0.
\end{equation}

Combining \eqref{eqn:gS-bound-0} and \eqref{eqn:S-lower} gives
\begin{equation}\label{eqn:gS-bound}
E_\mu^P\sqbk{g\crbk{S_{\theta_0}(\epsilon)} \1_C}
\le
\frac{4}{I_C(P\|P_0)}\sqbk{\beta P_\mu^P(T_C<\infty)+(\epsilon-l_{P,P_0})_+E_\mu^P\theta_0 \1_C}
+\kappa_C P_\mu^P(T_C<\infty).
\end{equation}
Substituting \eqref{eqn:gS-bound} into \eqref{eqn:sigma-reg-bound} yields
\begin{equation}\label{eqn:Esigma-pre}
E_\mu^P\sigma(\epsilon) \1_C
\le
\crbk{1+\frac{4(\epsilon-l_{P,P_0})_+}{I_C(P\|P_0)}}E_\mu^P\theta_0 \1_C
+\frac{4\beta P_\mu^P(T_C<\infty)}{I_C(P\|P_0)}
+\kappa_C P_\mu^P(T_C<\infty).
\end{equation}

Finally, define the within-class hitting bound
\[
H_C:=\max_{s\in C}E_s^P\theta_0<\infty.
\]
By the strong Markov property at $T_C$,
\begin{equation}\label{eqn:theta0-TC}
\begin{aligned}E_\mu^P\theta_0 \1_C
&\le
E_\mu^P T_C \1_C+H_C P_\mu^P(T_C<\infty)\\
&\leq \max_{s\in S}E_s^P T_C \1_C+H_C. 
\end{aligned}
\end{equation}
Combining \eqref{eqn:tau-le-sigma}, \eqref{eqn:Esigma-pre}, and \eqref{eqn:theta0-TC} gives \eqref{eqn:Etau-fin} after absorbing all $\rho$-independent factors into a constant $K_C$.
\end{proof}

\subsection{Construction of the Policy}\label{sec:construct_policy}

With these preparations, we construct the policy $\pi^*$ that achieves an $O(1)$ transient value under a weakly communicating assumption. The policy proceeds in exponentially growing epochs of length $L_j$ with exponentially smaller rejection probability levels $\rho_j$, as defined in \eqref{eqn:def_L_and_rho}. Within each epoch, $\pi^*$ initially runs $\Delta^*$ while computing a (shifted) mixture likelihood ratio $\Lambda_{t_j,t}$ in \eqref{eqn:def_mixture_LR_with_starting}. If the test rejects before the epoch ends, $\pi^*$ restarts and runs the reference RL policy $\pi_{\mathrm{RL}}$ from the rejection time until the end of the epoch; otherwise, it continues with $\Delta^*$ throughout. The epoch schedule is tuned so that the cumulative impact of occasional false rejections is summable, which is key to obtaining a constant-order transient value.

The intuition is that, if the adversary chooses $p^*$, rejections are rare and the policy behaves essentially like the optimal $\Delta^*$. If the adversary chooses a suboptimal $p$ and the dynamics under $p_{\Delta^*}$ are statistically distinguishable from those of $p^*_{\Delta^*}$, then rejection occurs sufficiently quickly (Theorem~\ref{thm:rej_time}), after which $\pi_{\mathrm{RL}}$ exploits the adversary's suboptimality, i.e., deviation from $p^*$.

With this in mind, we begin by specifying the setting. Note that if $\cP$ is weakly communicating, the optimal average reward $\alpha_p$, solving \ref{eqn:wc_Bellman_eqn} for each $p\in\cP$, is state-independent. Therefore, for any $\mu\in\cP(\cS)$, $\cP_\mu^* = \cP^*:= \set{p\in\cP:\alpha_p = \alpha^*}$ is independent of $\mu$.

Throughout the remainder of this paper, we will assume the following setting. 
\begin{assumption}\label{assump:wc_and_can_use_det}
    Assume that $\cP$ is weakly communicating and $\set{\delta_a:a\in A}\subseteq   \cQ$. 
\end{assumption}

Note that, by Corollary \ref{cor:piZX_RL}, under Assumption \ref{assump:wc_and_can_use_det}, we have
$\PiRL(\cQ,\cP)\neq \varnothing$. Hence, we fix a reference RL policy $\piRL\in\PiRL(\cQ,\cP)$. Next, to proceed with the construction of our policy of interest, we first present the following lemma, whose proof is deferred to Appendix \ref{sec:proof:lemma:wc_unichain_pi}. 

\begin{lemma}\label{lemma:wc_unichain_pi}
    Let $p$ be weakly communicating. Then there exists an average-reward optimal deterministic policy $\Delta:S\ra \set{\delta_a:a\in A}$--i.e., $\alpha(\mu,\Delta,p)=\alpha_p$ where $\alpha_p$ is the unique solution to \eqref{eqn:wc_Bellman_eqn}--such that $p_\Delta$ is unichain.
\end{lemma}

Assumption \ref{assump:wc_and_can_use_det} and Lemma \ref{lemma:wc_unichain_pi} implies that there exists an optimal policy $\Delta^*:S\ra \cQ$ such that \begin{equation}\label{eqn:choice_of_P0}
P_0(s'|s) := \sum_{a\in A}p^*(s'|s,a)\Delta^*(a|s)
\end{equation} is unichain. Moreover, one could pick $\Delta^*$ to be deterministic.  

\begin{remark}\label{rmk:irreducible}
Here, we do not restrict $\Delta^*$ to be deterministic. In particular, we will choose $\Delta^*$ so that the unique closed communicating class $C_0$ of $P_0$ is as large as possible. As discussed in Theorem \ref{thm:pistar_tv_bound}, it is beneficial if one can choose an optimal policy $\Delta^*$ such that $P_0$ is irreducible. In general, such a $\Delta^*$, if it exists, may need to be randomized. Indeed, one can construct weakly communicating average-reward MDP instances in which no deterministic optimal policy induces an irreducible Markov kernel on $S$, whereas a randomized optimal policy does induce an irreducible chain.

\end{remark}

Since we will implement an epoch-based strategy similar to that in Proposition \ref{prop:hp_RL_policy} to construct our $\pi^*\in\PiRL(\cQ,\cP)$, we update the definition of the mixture likelihood ratio process and rejection time with a starting time. Specifically, parallel to \eqref{eqn:def_mixture_LR} and \eqref{eqn:rej_time} and given $P_0$ in \eqref{eqn:choice_of_P0} and a product Dirichlet prior $\Pi$ as in Definition \ref{def:dirichlet-prior}, define the mixture likelihood process starting from time $m\geq 0$ as
\begin{equation}\label{eqn:def_mixture_LR_with_starting}
\Lambda_{m,n}
:=
\int \prod_{t=m}^{n-1}\frac{Q(X_{t+1}| X_t)}{P_0(X_{t+1}| X_t)} \Pi(dQ),
\end{equation} and rejection time  \begin{equation}\label{eqn:rej_time_starting}
\tau_{m,\rho}:=\inf\set{n\ge m+1: P_0(X_n| X_{n-1})=0\ \text{or}\ \Lambda_{m,n}\ge \frac1\rho}.
\end{equation}

We define the length and rejection parameter of the $j$th epoch as \begin{equation}\label{eqn:def_L_and_rho}
L_j := 2^j,
\quad
\rho_j := 2^{-\zeta j},
\quad j\ge 1.
\end{equation}
Here, $\zeta > 1$ is a tuning parameter. So, the starting time of the $j$th epoch is \begin{equation}\label{eqn:epoch_start_time}
    t_{j} = t_{j-1} + L_{j-1}
\end{equation}  for $j\ge 2$
with $t_1 = 0$.

\SetAlgorithmName{Policy}{policy}{List of Policies}
\begin{algorithm}[ht]
\caption{Intuitive description of policy $\pi^* $ with $O(1)$ transient value}
\label{policy:O1_TV_policy}
\DontPrintSemicolon

\KwIn{$\Delta^*$ and $P_0$ from \eqref{eqn:choice_of_P0}; product Dirichlet prior $\Pi$ with parameters $\gamma$;
reference RL policy $\pi_{\mathrm{RL}}$; epoch start times $\{t_j:j\ge 1\}$ from \eqref{eqn:epoch_start_time}; 
rejection parameters $\{\rho_j:j\ge 1\}$ from \eqref{eqn:def_L_and_rho}.}

Set the current time index {$t \gets 0$}

\For(\tcp*[f]{epochs}){$j \gets 1,2,\ldots$}{
    \Repeat(\tcp*[f]{test if $P = P_0$}){$P_0(X_{t}\mid X_{t-1})=0 \textbf{ or } \Lambda_{t_j,t} \ge 1/\rho_j$ \textbf{ or } $t = t_{j+1}$}{
    Choose action $A_{t} = \Delta^*\!\left(X_{t}\right)$ and transition to $X_{t+1}$\;
    Compute $\Lambda_{t_j,t+1}$ according to \eqref{eqn:def_mixture_LR_with_starting}\;
    Update time $t \gets t+1$\;
    }
    \If(\tcp*[f]{if reject before end of epoch, run RL}){$t< t_{j+1}$}{Run $\piRL$ until $t =  t_{j+1}$}
}

\end{algorithm}

The policy $\pi^*$ is informally outlined in Policy \ref{policy:O1_TV_policy}. It decomposes into epochs with lengths $\set{L_j:j\geq 1}$. Within each epoch, the policy has two phases: testing and RL. During the testing phase of the $j$th epoch, $\pi^*$ uses $\Delta^*$ to both explore and exploit, while computing the mixture likelihood ratio statistics and checking the rejection criterion.

If the rejection time is realized before time $t_{j+1}$, then the policy switches to the RL phase, in which it runs $\piRL$ until the end of the epoch, i.e., until $t=t_{j+1}$. Otherwise, the policy continues to use $\Delta^*$ until the end of the epoch.

\subsection{Transient Value Analysis}

In this section, we analyze the behavior of $\pi^*$. To verify the optimality of $\pi^*$ and study its transient value, we first provide a mathematically rigorous formulation of $\pi^*=(\pi_0^*,\pi_1^*,\ds)\in\PiH(\cQ)$, where $\pi_t^*:\bd H_t\ra\cQ$.

Define \begin{equation} \label{eqn:def_I_rej_indicator}I_{m,n,\rho}:= \1\set{\tau_{m,\rho}\leq n} .
\end{equation} Note that $(\Lambda_{m,n},I_{m,n,\rho}) $ are $\sigma(X_m,\ds,X_n)\subseteq  \cH_n$ measurable. So, through an abuse of notation, we consider them as measureable functions on $\bd H_n$; i.e, $(\Lambda_{m,n}(\omega),I_{m,n,\rho}(\omega)) = (\Lambda_{m,n}(h_n),I_{m,n,\rho}(h_n))$ where $h_n\in \bd H_n$ is such that $\omega = (s_0,a_0,\ds ,s_n,a_n,\ds) = (h_n,a_n,\ds)$. For each $n$, we also define the shift operator on $\bd H_n$ by $f_n^k(h_n) := (s_k,a_k,\ds,a_{n-1},s_n)$ for $k = 0,\ds ,n$

Following Policy \ref{policy:O1_TV_policy}, we define $\set{\pi^*_t:t\geq 0}$ as follows. For $t = t_j,\ds,t_{j+1}-1$ and any $h_t\in\bd H_t$ 
\begin{equation}\label{eqn:def_pi_star}
    \pi_t^*(\cd|h_t) = \begin{cases}
        \Delta^*(\cd|s_t), & \text{if}\ I_{t_j,t,\rho_j}(h_t) = 0;\\
        \piRL_{t-\tau }(\cd|f_t^{\tau}(h_t)), &\text{if}\ I_{t_j,t,\rho_j}(h_t) = 1;
    \end{cases}
\end{equation}
where $\tau = \tau_{t_j,\rho_j}(\omega_t)$ with $\omega_t = (h_t,a_t',s_{t+1}'\ds)\in\Omega$ for some $(a'_t,s_{t+1}'\ds)$. Note that, if $I_{t_j,t,\rho_j}(h_t) = 1$, then $\tau$ does not depend on $(a'_t,s_{t+1}'\ds)$. So, $\pi^*$ is indifferent to the choice of $(a'_t,s_{t+1}'\ds)$. 

With this definition, we present the main result of this section.

\begin{theorem}[Transient value bound for fixed-length epochs]\label{thm:pistar_tv_bound} Suppose Assumption \ref{assump:wc_and_can_use_det} is in force and $\cP^* = \set{p^*}$ is a singleton.  Construct $\pi^*$ as in \eqref{eqn:def_pi_star} with inputs specified in Policy \ref{policy:O1_TV_policy}. Also, assume that either of the following holds.
\begin{itemize}
    \item Identifiability: For $p\in\cP$ with $P = p_{\Delta^*}$, $P|_{C_0}\neq P_0|_{C_0}$. 
    \item Irreducibility: $P_0$ is irreducible. 
\end{itemize}

Then, for every \(\mu\in\cP(S)\) and $\zeta > 1$ used in \eqref{eqn:def_L_and_rho},
\begin{equation}\label{eqn:tv_lower_bound_main}
\mrm{TV}(\mu,\pi^*)
\ge
-\frac{2^{\zeta}}{2^{\zeta}-1}\spnorm{v^*}
-\frac{1}{2^{\zeta-1}-1},
\end{equation}
where $v^*$ is the unique (up-to a shift) solution to 
\begin{equation}\label{eqn:vstar_eval_eqn}
  v^*(s) =\sum_{a\in A} \Delta^*(a|s)r(s,a) - \alpha^* + \sum_{s'\in S}P_0(s'|s)v^*(s'). 
\end{equation}
In particular, for any $\epsilon > 0$, we can choose $\zeta$ sufficiently large so that $\mrm{TV}(\mu,\pi^*)\geq -\spnorm{v^*} - \epsilon$. 
\end{theorem}

\begin{remark}
    It is possible to achieve an $O(1)$ transient value without the identifiability or irreducibility assumptions. In this case, one can consider a more complex policy with an additional test statistic that rejects when the chain is not absorbed into $C_0$ quickly enough. However, without the identifiability or irreducibility assumptions, we believe it may not be possible to obtain a clean $-O(\spnorm{v^*})$ lower bound on the transient value, since weakly communicating alternative kernels with $C_0$ as their only closed communicating class under $\Delta^*$ can have arbitrarily large hitting times to $C_0$. So, the maximum expected hitting time could also enter the lower bound. 

    Therefore, in this paper, we opt to introduce these additional assumptions in order to obtain the clean result in Theorem \ref{thm:pistar_tv_bound}, with explicit bounds and clear algorithm design intuition.

    The assumption on the uniqueness of $p^*$ could potentially be removed, while maintaining an $O(1)$ transient value, if $|\cP^*|<\infty$ and we tailor the policy design to the specific adversarial environments $p^*\in\cP^*$. For instance, \citet{bubeck2013bounded_regret} shows that, in a multi-arm bandit setting where the adversarial environments are permutations of the same set of arm probabilities, i.e., $\cP^*=\cP$ is finite, there exists a policy that achieves $O(1)$ regret. However, this cannot be generalized to more general multi-arm bandit settings with continuum $\cP$, as evidenced by the classical $O(\log T)$ regret lower bound established by \citet{lai1985asymptotically}; see also Theorems 6 and 8 in \citet{bubeck2013bounded_regret}.

\end{remark}

\subsubsection{Proof of Theorem \ref{thm:pistar_tv_bound}}

\begin{proof}

We first note that, with unichain $P_0$, the existence and up-to-shift uniqueness of $v^*$ is standard for finite-state unichain average-reward models; see, e.g., \citep[Chs.~8--9]{puterman2014MDP}.

Recall the definition of transient value in \eqref{eqn:weighted_TV}. Within the setting of Theorem \ref{thm:pistar_tv_bound}, $\alpha^*(\mu) = \alpha^*$ for all $\mu\in\cP(\cS)$. So, for any fixed $p\in\cP$, we consider the pre-limit quantity
\begin{equation}\label{eqn:to_use_pre_lim_TV}
    E_\mu^{\pi^*,p} \sum_{t=0}^{T-1}[r(X_t,A_t) -\alpha^*]
\end{equation}

Fix an arbitrary adversarial kernel $p\in\cP$, denote
$P(s'|s) = p(s'|s,\Delta^*(s)), \quad s,s'\in S$.  
There are three disjoint cases: (1) $P = P_0$; (2) $P\neq P_0$, $P|_{C_0} \neq P_0|_{C_0}$; (3) $P\neq P_0$, $P|_{C_0} = P_0|_{C_0}$.
Clearly, the three cases cover any $p\in\cP$.  

Note that case 3 cannot happen if the identifiability condition hold. Moreover, if $P_0$ is irreducible,  $C_0 = S$ and  $P|_{C_0} = P_0|_{C_0}$, then we must have $P_0 = P$. So, under either identifiability or irreducibility, case 3 cannot happen. Therefore, we only need to analyze \eqref{eqn:to_use_pre_lim_TV} in the first two cases. 

For simplicity, let us denote the ending time of the $j$th epoch by $e_j = t_{j+1}-1$

\noindent\textbf{Case 1: $P = P_0$.}

\begin{lemma}\label{lemma:stopped_TV_sum_bd}
Let $\theta$ be a bounded $\cH_t$-stopping time, $P = P_0$, and $v^*$ be defined by \eqref{eqn:vstar_eval_eqn}. Then with $P = P_0$ for every $\mu\in\cP(\cS)$,
\begin{equation}\label{eqn:stopped_sum_identity}
E_\mu^{\Delta^*,p}\sqbk{\sum_{t=0}^{\theta-1}\crbk{r(X_t,A_t)-\alpha^*}}
=
E_\mu^{\Delta^*,p}\sqbk{v^*(X_0)-v^*(X_\theta)}.
\end{equation}
In particular,
\begin{equation}\label{eqn:stopped_sum_span_bd}
E_\mu^{\Delta^*,p}\sqbk{\sum_{t=0}^{\theta-1}\crbk{r(X_t,A_t)-\alpha^*}}
\ge -\spnorm{v^*}.
\end{equation}
\end{lemma}

First, we note that within each epoch, 
$$E_\mu^{\pi^*,p} \sum_{t=t_j}^{e_j}[r(X_t,A_t) - \alpha^*] 
\ge E_\mu^{\pi^*,p} \sum_{t=t_j}^{(\tau_{t_j,\rho_j}- 1)\wedge e_j }[r(X_t,A_t) - \alpha^*] - \underbrace{L_jP_\mu^{\pi^*,p}(\tau_{t_j,\rho_j} \leq e_j)}_{:=\psi_j}$$

Let $ J = \max\set{j\ge1:t_j\leq T}$. 
\begin{equation}\label{eqn:to_bd_pre_lim_TV}
    \begin{aligned}
    &E_\mu^{\pi^*,p} \sum_{t=0}^{T-1}[r(X_t,A_t) - \alpha^*]\\
    &\ge  \sum_{j=1}^{J-1}\crbk{E_\mu^{\pi^*,p} \sum_{t=t_j}^{(\tau_{t_j,\rho_j}-1)\wedge e_j}[r(X_t,A_t) - \alpha^*] - \psi_j}\\
    &\quad + E_\mu^{\pi^*,p}\sqbk{\sum_{t=t_{J}}^{(\tau_{t_{J},\rho_{J}}-1)\wedge (T-1)} [r(X_t,A_t) - \alpha^*] }- \psi_J\\
    &\geq E_\mu^{\pi^*,p} \sqbk{ \sum_{j=1}^{J}\sum_{t=t_j}^{(\tau_{t_j,\rho_j}\wedge t_{j+1}\wedge  T)- 1}[r(X_t,A_t) - \alpha^*]}-\sum_{j=1}^J \psi_j
\end{aligned} 
\end{equation}

To lower bound the last line of \eqref{eqn:to_bd_pre_lim_TV} we first start with bounding $\psi_j$.  Note that for $t_{j}\leq t\leq e_j$ and any $h'_{t_j} = (s_0',a_0',\ds,s_{t_j}')\in\bd H_{t_j}$.
\begin{equation}\label{eqn:to_ref_epoch_quantity_derivation}\begin{aligned}
&P_\mu^{\pi^*,p} \crbk{\tau_{t_j,\rho_j} > e_j,H_{t_j} = h'_{t_j}} \\
&= \sum_{h_{e_j} \in \bd H_{e_j}}\1\set{h_{t_j}= h'_{t_j},I_{t_j,e_j,\rho_j}(h_{e_j}) = 0}\mu(s_0)\prod_{t=0}^{e_j-1}\pi_t^*(a_t|h_t)p(s_{t+1}|s_t,a_t)\\
&= \sum_{h_{e_j} \in \bd H_{e_j}}P_{\mu}^{\pi^*,p}(H_{t_j} = h_{t_j})\1\set{h_{t_j}= h'_{t_j},I_{t_j,e_j,\rho_j}(h_{e_j}) = 0}\prod_{t=t_j}^{e_j-1}\pi_t^*(a_t|h_t)p(s_{t+1}|s_t,a_t)\\ 
&\stackrel{(i)}{=} P_{\mu}^{\pi^*,p}(H_{t_j} = h_{t_j}')\sum_{h_{e_j} \in \bd H_{e_j}}\1\set{h_{t_j}= h'_{t_j},I_{t_j,e_j,\rho_j}(h_{e_j}) = 0}\prod_{t=t_j}^{e_j-1}\Delta^*(a_t|s_t)p(s_{t+1}|s_t,a_t)\\ 
&\stackrel{(ii)}{=} P_{\mu}^{\pi^*,p}(H_{t_j} = h_{t_j}')\sum_{s_{t_j},a_{t_j},\ds,s_{e_j}}\1\set{I_{t_j,e_j,\rho_j}((h'_{t_j},a_{t_j},\ds,s_{e_j})) = 0}\delta_{s_{t_j}'}(s_{t_j})\prod_{t=t_j}^{e_j-1}\Delta^*(a_t|s_t)p(s_{t+1}|s_t,a_t)\\ 
&\stackrel{(iii)}{=} P_{\mu}^{\pi^*,p}(H_{t_j} = h'_{t_j})P^{P_0}_{s_{t_j}'}(\tau_{\rho_j}> L_j-1)
\end{aligned}
\end{equation}
where we use the following simplified notation that $h_{e_j} = (h_t,a_t,\ds ,s_{t_{j+1} -1})$ for all $t\leq e_j$. Note that $(i)$ follows from the facts that
\begin{itemize}
    \item $I_{t_j,e_j,\rho_j}(h_{e_j}) = 0$ implies $I_{t_j,t,\rho_j}(h_t) = 0$ for all $t\leq e_j$; i.e., if rejection didn't happen before $e_j$ then it didn't happen before any $t\leq e_j$. 
    \item $\pi^*_t$ defined by \eqref{eqn:def_pi_star} satisfies $\1\set{I_{t_j,t,\rho_j}(h_{t}) = 0}\pi^*_t(a_t|h_t) = \1\set{I_{t_j,t,\rho_j}(h_{t}) = 0}\Delta^*(a_t|h_t)$,
\end{itemize} 
$(ii)$ follows from summing over $s_0,a_0\ds,s_{e_{j-1}},a_{e_{j-1}}$, and $(iii)$ reverts the first equality and notes that the state process has transition kernel $P(s'|s) = P_0(s'|s) =  \sum_{a\in A}p(s'|s,a)\Delta^*(a|s)$. This implies that
\begin{equation}\label{eqn:to_use_cond_rej_in_epoch_prob}
    P_\mu^{\pi^*,p} \crbkcond{\tau_{t_j,\rho_j} > e_j}{\cH_{t_j}} = P_{X_{t_j}}^{P_0}(\tau_{\rho_j} > L_j-1)
\end{equation}
a.s.$P_\mu^{\pi^*,p}$. 

Therefore, by \eqref{eqn:to_use_cond_rej_in_epoch_prob} and Theorem \ref{thm:rej_time}, 
\begin{equation}\label{eqn:to_use_rej_in_epoch_prob_bd}\begin{aligned}
\psi_j&=L_jP_\mu^{\pi^*,p}(\tau_{t_j,\rho_j} \leq e_j) \\
&= L_jE_\mu^{\pi^*,p}P_\mu^{\pi^*,p} \crbkcond{\tau_{t_j,\rho_j}\leq e_j}{\cH_{t_j}}\\
&\leq L_jE_\mu^{\pi^*,p}P_{X_{t_j}}^{P_0} \crbk{\tau_{\rho_j}< \infty}\\
&\leq L_j\rho_j.
\end{aligned}
\end{equation}

For simplicity, let us define for all $j\leq J$
\begin{equation}\label{eqn:def_T_j}
T_j = \tau_{t_j,\rho_j}\wedge t_{j+1}\wedge T.
\end{equation}
Going back to \eqref{eqn:to_bd_pre_lim_TV}, we note that the sum
$$ \sum_{j=1}^{J}\sum_{t=t_j}^{(\tau_{t_j,\rho_j}\wedge t_{j+1}\wedge  T)- 1}[r(X_t,A_t) - \alpha^*] = \sum_{j=1}^{J}\sum_{t=t_j}^{T_j- 1}[r(X_t,A_t) - \alpha^*]$$
can be regrouped into cycles, where each cycle represents the period between the start-of-the-previous and the end-of-the-next deployments of $\Delta^*$ $t = 0,\ds,T-1$. Specifically, 
we define $\chi_1'= 0$
$$ \chi_k := \inf\set{t > \chi_{k}': t = \tau_{t_j,\rho_j} < e_j },\quad \chi_{k+1}' := \inf\set{t_j > \chi_{k}: j\geq 2 },\quad k\ge 1$$ representing the $k$th time that $\pi^*$ deploys and ends $\piRL$. It is not hard to see that, when finite, $\chi_k'<\chi_k  < \chi_{k+1}'$ and $\chi_k'\geq t_k$. Also, they are $\cH_t$-stopping times, as $\1\set{\chi_{k} \leq t}$ and $ \1\set{\chi_{k}' \leq t}$ are determined by $(s_0,s_1,\ds,s_t)$. Moreover, 
$$ \sum_{j=1}^{J}\sum_{t=t_j}^{T_j- 1}[r(X_t,A_t) - \alpha^*]  = \sum_{k=1}^{J} \sum_{t=\chi_k'}^{(\chi_{k}\wedge T)-1}[r(X_t,A_t) - \alpha^*].$$

Therefore, 
$$\begin{aligned}
        &E_\mu^{\pi^*,p}\sqbk{\sum_{t=\chi_k'}^{(\chi_{k}\wedge T)-1}[r(X_t,A_t) - \alpha^*]}\\
        &=  E_\mu^{\pi^*,p}\sqbk{\1\set{\chi_k'\leq t_J}\sum_{t=\chi_k'}^{(\chi_{k}\wedge T)-1}[r(X_t,A_t) - \alpha^*]}\\
        &= E_\mu^{\pi^*,p}\1\set{\chi'_k\leq t_J}E_\mu^{\pi^*,p}\sqbkcond{\sum_{t=\chi_k'}^{(\chi_{k}\wedge T)-1}[r(X_t,A_t) - \alpha^*]}{\cH_{\chi'_k}}\\
        &= E_\mu^{\pi^*,p}\sqbk{\1\set{\chi'_k\leq t_J}g(X_{\chi_k'},T-\chi_k') }
    \end{aligned}
$$
where, by Lemma \ref{lemma:stopped_TV_sum_bd},
$$g(x,y):=E_{x}^{\Delta^*,p}{\sum_{t=0}^{(\chi_1\wedge y)-1}[r(X_t,A_t) - \alpha^*]}\geq -\spnorm{v^*}.$$
This is because, by the definition of $\pi^*$, in between the times $\chi'_k$ and $\chi_{k}$, the controlled Markov chain $(X,A)$ under $E_\mu^{\pi^*,p}$ has the same action and transition probabilities as that under $(\Delta^*,p)$. 

Let $N_T$ denote the number of rejections before time $T$. Then, 
$$ \sum_{k=1}^{J}\1\set{\chi_k'\leq t_J}\le  1 +\sum_{j=1}^J\1\set{ \tau_{t_j,\rho_j} < e_j}  = 1+N_T.$$ Hence, 
\begin{equation}\label{eqn:to_use_prelim_E_TV_bd}
\begin{aligned}
E_\mu^{\pi^*,p}\sqbk{\sum_{k=1}^{J}\sum_{t=\chi_k'}^{(\chi_{k}\wedge T)-1}[r(X_t,A_t) - \alpha^*]}&\geq -\spnorm{v^*}E_\mu^{\pi^*,p}[N_T+1]\\
    &= -\spnorm{v^*} - \spnorm{v^*}\sum_{i=1}^JP_\mu^{\pi^*,p}(\tau_{t_j,\rho_j} < e_j)\\
    &\geq -\spnorm{v^*} - \spnorm{v^*}\sum_{j=1}^J E_\mu^{\pi^*,p} P_\mu^{\pi^*,p}(\tau_{t_j,\rho_j}  \leq  e_j|\cH_{t
    _j})\\
    &\geq -\spnorm{v^*}\crbk{1+\sum_{j=1}^J\rho_j}
\end{aligned}
\end{equation}
where the last inequality follows from \eqref{eqn:to_use_cond_rej_in_epoch_prob} and \eqref{eqn:to_use_rej_in_epoch_prob_bd}. 

Therefore, combining \eqref{eqn:to_bd_pre_lim_TV},  \eqref{eqn:to_use_rej_in_epoch_prob_bd}, and \eqref{eqn:to_use_prelim_E_TV_bd}, we conclude that 
\begin{equation}\label{eqn:to_use_P_eq_P_0_TV_nd}
\begin{aligned}
    E_\mu^{\pi^*,p} \sum_{t=0}^{T-1}[r(X_t,A_t) - \alpha^*] &\geq -\spnorm{v^*}\crbk{1+\sum_{i=1}^J\rho_j} - \sum_{i=1}^{J}\psi_j\\
    &\geq  -\spnorm{v^*}\crbk{1+\sum_{i=1}^\infty\rho_j} - \sum_{i=1}^{\infty}L_j\rho_j\\
    &\geq -\spnorm{v^*}\crbk{1+\sum_{i=1}^\infty 2^{-\zeta j}} - \sum_{i=1}^{\infty}2^{-(\zeta-1)j}\\
    &= -\frac{2^{\zeta}}{2^{\zeta}-1}\spnorm{v^*}
-\frac{1}{2^{\zeta-1}-1}. 
\end{aligned}
\end{equation}

\noindent\textbf{Case 2: $P\neq P_0$, but $P|_{C_0} \neq P_0|_{C_0}$. }

In this case, we will show that 
\begin{equation}\label{eqn:to_show_infty_TV}
    \liminf_{T\ra\infty}E_\mu^{\pi^*,p}\sum_{t=0}^{T-1}\sqbk{r(X_t,A_t)-\alpha^*}  = \infty
\end{equation}

We introduce the following lemma that describes an important property of the RL policy when $p\neq p^*$. The proof of Lemma \ref{lemma:RL_tv_lb} is deferred to Appendix \ref{sec:proof:lemma:RL_tv_lb}. 
\begin{lemma}\label{lemma:RL_tv_lb}
There exists $T_0$ such that for all $s\in S$,  $$E_s^{\piRL,p}\sum_{t=0}^{T-1}[r(X_t,A_t) - \alpha^*] \geq \frac{(\alpha_p-\alpha^*)T}{2} - 2T_0.$$
\end{lemma}

We now lower bound the pre-limit transient value. Recall the definition of $T_j$ in \eqref{eqn:def_T_j}. We observe that
\begin{equation}\label{eqn:to_use_case2_TV_epoch_decomp}
\begin{aligned}
&E_\mu^{\pi^*,p}\sum_{t=0}^{T-1}\sqbk{r(X_t,A_t)-\alpha^*}\\
&=
\sum_{j=1}^{J}E_\mu^{\pi^*,p}\sqbk{\sum_{t=t_j}^{T_j-1}\sqbk{r(X_t,A_t)-\alpha^*} + \sum_{t=T_j}^{ (t_{j+1}\wedge T)-1}\sqbk{r(X_t,A_t)-\alpha^*}}
\end{aligned}
\end{equation}

We separately analyze the two terms within the expectation. First, note that for $j\leq J$,
$$\sum_{t=t_j}^{T_j-1}\sqbk{r(X_t,A_t)-\alpha^*}  \geq -(T_j- t_j).$$
Since $T_j = \tau_{\rho_j}\wedge t_{j+1}\wedge T$ is bounded, by a similar derivation as in \eqref{eqn:to_ref_epoch_quantity_derivation},
\begin{equation}\label{eqn:to_use_epoch_j_E_tau_bd}
    E_\mu^{\pi^*,p}\sqbkcond{T_j}{\cH_{t_j}}= E_{X_{t_j}}^{\Delta^*,p}[\tau_{\rho_j}\wedge L_j \wedge( T-t_j)] + t_j \leq K\crbk{\log\frac{1}{\rho_j} + 1}+t_j
\end{equation}
a.s.$P_\mu^{\pi^*,p}$ for some $K$ that only depend on $(P,P_0,\gamma)$. Noting that in this case $P|_{C_0}\neq P_0|_{C_0}$, the last inequality in \eqref{eqn:to_use_epoch_j_E_tau_bd} follows from Theorem \ref{thm:rej_time}. Therefore, taking another expectation on \eqref{eqn:to_use_epoch_j_E_tau_bd}, we obtain
\begin{equation}\label{eqn:to_use_before_RL_bd}
\begin{aligned}E_\mu^{\pi^*,p}\sum_{t=t_j}^{T_j-1}\sqbk{r(X_t,A_t)-\alpha^*}&\geq -K\crbk{\log\frac{1}{\rho_j}+1}. 
\end{aligned}
\end{equation}

Next, by the construction of $\pi^*$ in  \eqref{eqn:def_pi_star} and a similar argument as in \eqref{eqn:to_ref_epoch_quantity_derivation}, for all $j\leq J$ and $\omega' = (s_0',a_0',s_1,'\ds)\in\Omega$, 
$$\begin{aligned}
&E_\mu^{\pi^*,p}\sqbk{\1\set{H_{T_j} = H_{T_j(\omega')}(\omega')}\sum_{t=T_j}^{ (t_{j+1}\wedge T)-1}\sqbk{r(X_t,A_t)-\alpha^*}}\\
&= P_\mu^{\pi^*,p}\crbk{H_{T_j} =H_{T_j(\omega')}(\omega')} E_{{s'}_{T_j(\omega')}}^{\piRL,p}\sqbk{\sum_{t=0}^{ (t_{j+1}\wedge T) - T_j(\omega')-1}\sqbk{r(X_t,A_t)-\alpha^*}}. 
\end{aligned}$$
Here, we note that $T_j = \tau_{\rho_j,t_j}\wedge t_{j+1}\wedge T$ is bounded, so $P_\mu^{\pi^*,p}\crbk{H_{T_j} = H_{T_j(\omega')}(\omega')} > 0$. Thus, 
\begin{align*}E_\mu^{\pi^*,p}\sqbkcond{\sum_{t=T_j}^{ (t_{j+1}\wedge T)-1}\sqbk{r(X_t,A_t)-\alpha^*}}{\cH_{T_j}} &= E_{{X}_{T_j}}^{\piRL,p}\sqbk{\sum_{t=0}^{ (t_{j+1}\wedge T)- y-1}\sqbk{r(X_t,A_t)-\alpha^*}}\Bigg|_{y = T_j}\\
&\geq   \frac{(\alpha_p-\alpha^*)(t_{j+1}\wedge T - T_j)}{2} - 2T_0. 
\end{align*}

Taking expectation, we obtain
\begin{equation}\label{eqn:to_use_cycle_RL_period}
\begin{aligned}
E_\mu^{\pi^*,p}\sum_{t=T_j}^{ (t_{j+1}\wedge T)-1}\sqbk{r(X_t,A_t)-\alpha^*} &\ge\frac{(\alpha_p-\alpha^*)(t_{j+1}\wedge T - E_\mu^{\pi^*,p} T_j)}{2} - 2T_0\\
&\geq \frac{(\alpha_p-\alpha^*)}{2} (t_{j+1}\wedge T - t_j)- \frac{(\alpha_p-\alpha^*)}{2}K\crbk{\log \frac{1}{\rho_j}+1} - 2T_0
\end{aligned}
\end{equation}
where the last inequality follows from \eqref{eqn:to_use_epoch_j_E_tau_bd}. 

Therefore, combining \eqref{eqn:to_use_case2_TV_epoch_decomp}, \eqref{eqn:to_use_before_RL_bd}, \eqref{eqn:to_use_cycle_RL_period}, and recall that $\rho_j = 2^{-\zeta j}$, we conclude that 
\begin{align*}&E_\mu^{\pi^*,p}\sum_{t=0}^{T-1}\sqbk{r(X_t,A_t)-\alpha^*} \\
&\geq -2T_0 J+  \frac{(\alpha_p-\alpha^*)}{2}\crbk{\sum_{j=1}^J(t_{j+1}\wedge T - t_j)  } -\crbk{\frac{(\alpha_p-\alpha^*)}{2} + 1}K\crbk{1+ \sum_{j=1}^J\zeta j\log (2) }\\
&\geq -2T_0J +  \frac{(\alpha_p-\alpha^*)}{2}T - K\crbk{2+ \log (2)\zeta J(J+1) }
\end{align*}
where the last inequality used that $0 < \alpha_p - \alpha^*\leq 2$. 
Recall that $J = \max\set{j\ge 1:t_j\leq T}$. So, by \eqref{eqn:epoch_start_time} $$T>t_{J} = \sum_{j=1}^{J-1}L_j = 2^{J}-2$$
So, $J = O(\log_2 (T+2))$. Therefore, 
\begin{align*}\liminf_{T\ra\infty}E_\mu^{\pi^*,p}\sum_{t=0}^{T-1}\sqbk{r(X_t,A_t)-\alpha^*} \geq   \liminf_{T\ra\infty}\sqbk{ \frac{(\alpha_p-\alpha^*)}{2}T-O(\log_2^2(T+2))} = +\infty,
\end{align*}
showing \eqref{eqn:to_show_infty_TV}. 

\textbf{Concluding Theorem \ref{thm:pistar_tv_bound}: }

To conclude the proof of Theorem \ref{thm:pistar_tv_bound},
we recall that $\set{p\in\cP: P \neq  P_0, P|_{C_0} = P_{0|C_0}} = \varnothing$. So,
\begin{align*}\mrm{TV}(\mu,\pi^*)&= \inf_{p\in\cP}\liminf_{T\ra\infty}E_\mu^{\pi^*,p}\sum_{t=0}^{T-1}\sqbk{r(X_t,A_t)-\alpha^*} \\
&=\min\bigg\{ \inf_{p\in\cP, P = P_0}\liminf_{T\ra\infty}E_\mu^{\pi^*,p}\sum_{t=0}^{T-1}\sqbk{r(X_t,A_t)-\alpha^*},\\
&\qquad\quad \quad \inf_{p\in\cP: P \neq  P_0, P|_{C_0} \neq P_{C_0}}\liminf_{T\ra\infty}E_\mu^{\pi^*,p}\sum_{t=0}^{T-1}\sqbk{r(X_t,A_t)-\alpha^*}\bigg\}\\
&\stackrel{(i)}{=} \min\set{-\frac{2^{\zeta}}{2^{\zeta}-1}\spnorm{v^*}
-\frac{1}{2^{\zeta-1}-1},+\infty}\\
&= -\frac{2^{\zeta}}{2^{\zeta}-1}\spnorm{v^*}
-\frac{1}{2^{\zeta-1}-1}. 
\end{align*}
where $(i)$ follows from \eqref{eqn:to_use_P_eq_P_0_TV_nd} and \eqref{eqn:to_show_infty_TV}. This completes the proof.
\end{proof}

\subsection*{Acknowledgments}
We thank Julien Grand-Cl{\'e}ment for valuable input and comments. Nian Si’s research is supported in part by the Early Career Scheme [Grant 	26210125] from the Hong Kong Research Grants Council. Shengbo Wang would like to thank the Isaac Newton Institute for Mathematical Sciences, Cambridge, for support and hospitality during the programme Bridging Stochastic Control 
And Reinforcement Learning, where work on this paper was undertaken. This work was supported by EPSRC grant EP/V521929/1.

\bibliographystyle{apalike}
\bibliography{proof_bibs,DR_MDP,mdps}

\newpage
\appendix
\appendixpage
\section{Proof of Auxiliary Lemmas}
\subsection{Proof of Lemma \ref{lemma:exist_worst_case_kernel_WC}}\label{sec:proof:lemma:exist_worst_case_kernel_WC}
\begin{proof}
Recall that under weak communication, \(\alpha_p(s)=\alpha_p\) for all \(s\in S\), where \((\alpha_p,v_p)\) solves \eqref{eqn:wc_Bellman_eqn} and $\alpha_p$ is unique. 
Moreover, by Corollary \ref{cor:piZX_RL},
$\alpha^*(\mu)=\alpha^*=\inf_{p\in\cP}\alpha_p,
$ for all $\mu\in\cP(\cS).$
Hence, since \(\cP\) is compact, it suffices to prove that \(p\ra \alpha_p\) is continuous on \(\cP\).

For \(p,q\in\cP\), define
\[
d(p,q):=\max_{s\in S,a\in A} d_{\mrm{TV}}\bigl(p(\cdot|s,a),q(\cdot|s,a)\bigr),
\]
and, for \(v:S\to\R\),
\[
\cB_p[v](s):=\max_{a\in A}\set{r(s,a)+\sum_{s'\in S}p(s'|s,a)v(s')}.
\]
Since
\[
\abs{\sum_{s'\in S}[q(s'|s,a)-p(s'|s,a)]v(s')}
\le d(p,q)\spnorm{v},
\]
we have, for every \(s\in S\),
\[
\cB_q[v_p](s)\le \cB_p[v_p](s)+d(p,q)\spnorm{v_p}
= v_p(s)+\beta_{p,q}.
\]
where $\beta_{p,q}:=\alpha_p+d(p,q)\spnorm{v_p}$ and we use \eqref{eqn:wc_Bellman_eqn} that $\cB_p[v_p] -\alpha_p=v_p(s)$. 
Thus \(v_p\) satisfies the Bellman inequality for the controlled kernel \(q\). 

By the comparison theorem for average-reward MDPs \citep[Theorem 9.1.2]{puterman2014MDP},
\[
\alpha_q\le \beta_{p,q}
=
\alpha_p+d(p,q)\spnorm{v_p}.
\]
Exchanging the roles of \(p\) and \(q\) gives
\[
\alpha_p\le \alpha_q+d(p,q)\spnorm{v_q}.
\]
Hence
\begin{equation}\label{eq:alpha_local_continuity_bound}
|\alpha_q-\alpha_p|
\le d(p,q)\max\{\spnorm{v_p},\spnorm{v_q}\}.
\end{equation}

It remains to show that \(\spnorm{v_q}\) is locally bounded around each \(p\in\cP\). Fix \(p\in\cP\), and let \(C_p\) be the communicating class in the weak-communication decomposition of \(p\). Write \(\tau_B\) for the first hitting time of a set \(B\subseteq S\).

First, we control the entrance time into \(C_p\) under arbitrary stationary deterministic policies \(\pi\in\Pi_{\mrm{SD}}\). Let \(Q_{q,\pi}\) be the restriction of the transition matrix under \((q,\pi)\) to \(S\setminus C_p\). Since every state in \(S\setminus C_p\) is transient under every stationary policy for \(p\), the matrix \(I-Q_{p,\pi}\) is invertible, and by a first transition argument
\[
[(I-Q_{p,\pi})^{-1} e](s)  = E_s^{\pi,p}\tau_{C_p} \]
for all $s\in S\setminus C_p$, where $e$ is the the vector of all ones. 
As \(q\ra Q_{q,\pi}\) is continuous and \(\Pi_{\mrm{SD}}\) is finite, there exist \(\delta_p^{(1)}>0\) and \(M_p<\infty\) such that
\begin{equation}\label{eq:enter_Cp_uniform_bound}
\sup_{q:d(p,q)\le \delta_p^{(1)}}
\max_{\pi\in\Pi_{\mrm{SD}}}\max_{s\in S\setminus C_p}
E_s^{\pi,q}\tau_{C_p}
\le M_p.
\end{equation}

Next, we control the hitting time of any prescribed state in \(C_p\) under a suitable stationary policy. Fix \(s'\in C_p\), and consider the directed graph on \(S\) with an edge \(x\to y\) whenever \(p(y|x,a)>0\) for some \(a\in A\). Every state \(x\in S\) has a directed path to \(s'\): this is clear for \(x\in C_p\); if \(x\notin C_p\), then it is transient and hence some $y\in C_p$ is reachable from \(x\). But $y$ can reach any $s\in C_p$ by the definition of $C_p$. So, $x$ can reach $s'$ through $y$. 

Let \(\ell_{s'}(x)\) be the graph distance from \(x\) to \(s'\). For each \(x\neq s'\), choose a successor \(y_{x,s'}\) and an action \(a_{x,s'}\in A\) such that
\[
\ell_{s'}(y_{x,s'})=\ell_{s'}(x)-1,
\qquad
p(y_{x,s'}|x,a_{x,s'})>0.
\]
Define a stationary deterministic policy \(\pi^{ s'}\in\Pi_{\mrm{SD}}\) by \(\pi^{ s'}(x)=a_{x,s'}\) for \(x\neq s'\) and take arbitrary action at \(s'\). Under \(\pi^{ s'}\), from every \(x\in S\) there is a path to \(s'\) of length at most
\[
m_p(s'):=\max_{x\in S}\ell_{s'}(x)<\infty
\]
and with strictly positive probability under $(\pi^{ s'},p)$. Since there are only finitely many such designated paths, their probabilities admit a strictly positive minimum:
\[
\eta_p(s')
:=\min_{x\in S}
P_x^{\pi^{ s'},p}(\tau_{s'}\le m_p(s'))
>0.
\]
By continuity of hitting probabilities in $q$, there exists \(\delta_p^{(2)}(s')>0\) such that, whenever \(d(p,q)\le \delta_p^{(2)}(s')\), $$P_x^{\pi^{ s'},q}(\tau_{s'}\le m_p(s'))\ge \eta_p(s')/2,$$
for all $x\in S.$
By the Markov property, this implies the geometric tail bound
\[
P_x^{\pi^{ s'},q}(\tau_{s'}>km_p(s'))\le (1-\eta_p(s')/2)^k,
\qquad k\ge 0.
\]
Therefore, let $\delta^{(2)}_p:= \min_{s'\in C_p}\delta_p^{(2)}(s')$
\begin{equation}\label{eq:hit_state_uniform_bound}
\sup_{q:d(p,q)\le \delta_p^{(2)}}
\max_{s\in S,s'\in C_p}E_s^{\pi^{ s'},q}\tau_{s'}
\le \max_{s'\in C_p}\frac{2m_p(s')}{\eta_p(s')}
=:L_p.
\end{equation}

Now fix \(q\in\cP\) with
\[
d(p,q)\le \delta_p:=\min\{\delta_p^{(1)},\delta_p^{(2)}\},
\]
and choose
\[
\bar s=\bar s(q)\in\arg\max_{s\in C_p} v_q(s).
\]
Let \(\pi_q^*\in\Pi_{\mrm{SD}}\) be a stationary deterministic policy attaining the maximum in the Bellman equation for \(q\). Since \(0\le r\le 1\), we also have \(0\le \alpha_q\le 1\), hence \(-1\le r-\alpha_q\le 1\). Moreover the Bellman equation implies that 
$$G_n:=v_q(X_{n}) + \sum_{t=0}^{n-1}[r(X_t,\pi_q^*(X_t))-\alpha_q
]$$ is a $P_s^{\pi^*_q,q}$ martingale.

Since $\delta_p \leq \delta_p^{(1)}$, $P_s^{\pi^*_q,q}(\tau_{C_p} < \infty) = 1$. Therefore, stopping at \(\tau_{C_p}\wedge n\) and letting \(n\uparrow\infty\), we obtain for every \(s\in S\),
\[
v_q(s)
=
E_s^{\pi_q^*,q}\!\left[
\sum_{t=0}^{\tau_{C_p}-1}\bigl(r(X_t,\pi_q^*(X_t))-\alpha_q\bigr)
+v_q(X_{\tau_{C_p}})
\right]
\le v_q(\bar s)+E_s^{\pi_q^*,q}\tau_{C_p},
\]
because \(X_{\tau_{C_p}}\in C_p\) and \(v_q(X_{\tau_{C_p}})\le v_q(\bar s)\). Therefore, by \eqref{eq:enter_Cp_uniform_bound},
\begin{equation}\label{eq:upper_part_span_bound}
\max_{s\in S}v_q(s)\le v_q(\bar s)+M_p.
\end{equation}

On the other hand, for any \(s\in S\), apply the same arguments to the Bellman inequality
\[
v_q(s)\ge r(s,\pi^{\bar s}(s))-\alpha_q+\sum_{s'\in S}q(s'|s,\pi^{\bar s}(s))v_q(s')
\]
and the stopping time \(\tau_{\bar s}\wedge n\) and letting \(n\uparrow\infty\) yields
\[
v_q(s)
\ge
E_s^{\pi^{\bar s},q}\!\left[
\sum_{t=0}^{\tau_{\bar s}-1}\bigl(r(X_t,\pi^{\bar s}(X_t))-\alpha_q\bigr)
+v_q(\bar s)
\right]
\ge v_q(\bar s)-E_s^{\pi^{\bar s},q}\tau_{\bar s}.
\]
Hence, by \eqref{eq:hit_state_uniform_bound},
\begin{equation}\label{eq:lower_part_span_bound}
v_q(\bar s)-\min_{s\in S}v_q(s)
\le
\max_{s\in S}E_s^{\pi^{\bar s},q}\tau_{\bar s}
\le L_p.
\end{equation}
Combining \eqref{eq:upper_part_span_bound} and \eqref{eq:lower_part_span_bound}, we conclude that
\[
\sup_{q:d(p,q)\le \delta_p}\spnorm{v_q}\le M_p+L_p<\infty.
\]

Returning to \eqref{eq:alpha_local_continuity_bound}, for \(d(p,q)\le \delta_p\),
\[
|\alpha_q-\alpha_p|
\le d(p,q)\bigl(\spnorm{v_p}+M_p+L_p\bigr).
\]
Thus \(q\to p\) implies \(\alpha_q\to \alpha_p\). Therefore \(p\ra \alpha_p\) is continuous on \(\cP\). Since \(\cP\) is compact, there exists \(p^*\in\cP\) such that $\alpha_{p^*}=\min_{p\in\cP}\alpha_p=\alpha^*.$
Therefore, Assumption \ref{assump:existence_of_pstar} holds.
\end{proof}

\subsection{Proof of Lemma \ref{lemma:liminf_of_weighted_seq}}\label{sec:proof:lemma:liminf_of_weighted_seq}
\begin{proof}
Suppose, to the contrary, that
\[
\liminf_{T\to\infty} w(T)g(T) > 0.
\]

Since $w(T)>0$, there exists $\epsilon>0$ and $T_0\in\N$ such that for all $T\ge T_0$,
$$w(T)g(T)\ge \epsilon\implies g(T)\ge \frac{\epsilon}{w(T)}.$$
Because $w(T)\to 0$, we have $\varepsilon/w(T)\to \infty$. Hence $\liminf_{T\to\infty} g(T)=+\infty,$ contradicting the assumption that $\liminf_{T\to\infty} g(T)\le C<\infty$.

Therefore, we conclude that $\liminf_{T\to\infty} w(T)g(T)\le 0$, as claimed.
\end{proof}

\subsection{Proof of Lemma~\ref{lemma:dirichlet-prior}}\label{sec:proof:lemma:dirichlet-prior}
\begin{proof}
Fix $s\in S$ and write $p(\cd):=P(\cd | s)$. Let $\nu$ be the uniform distribution on $S$, i.e., $\nu(s'):=1/\card( S)$ for all  $s'\in S$. Define
\[
q^*(s'):=e^{-\epsilon}p(s')+(1-e^{-\epsilon})\nu(s'),\quad \forall s'\in S.
\]
Then $q^*\in\cP(\cS)$ and $q^*(s')>e^{-\epsilon}p(s')$ for all $s'\in S$, so $q^*$ lies in the relative interior of
\[
U_{s}(\epsilon,p):=\set{q\in\cP(\cS): q(s')\ge e^{-\epsilon}p(s'), \forall s'\in S}.
\]
Hence, there exists $\delta>0$ such that $$\set{q\in\cP(\cS):\max_{s'\in S}{|q(s')-q^*(s')|}<\delta}\subseteq   U_{s}(\epsilon,p).$$

Because $\gamma(s' | s)>0$ for all $s'$, the Dirichlet law $\Pi_s = \mrm{Dirichlet}(\gamma(\cd | s))$ has a density proportional to $$d\Pi_s(q) \propto \prod_{s'\in S}q(s')^{\gamma(s' | s)-1}dq$$ on $\cP(\cS)$, which is finite and strictly positive in a neighborhood of $q^*$. Therefore, $\Pi_s\crbk{U_{s}(\epsilon,p)}>0.$

Note that $$U(\epsilon,P) = \set{Q:Q(\cd|s)\in U_s(\epsilon,P(\cd|s))} = \bigtimes_{s\in S} U_{s}(\epsilon,P(\cd|s)).$$ Since the rows are independent under $\Pi$, we conclude
\[
\Pi(U(\epsilon,P))
= \prod_{s\in S}\Pi_s\crbk{U_{s}(\epsilon,P(\cd|s))}
>0.
\]
This completes the proof.
\end{proof}

\subsection{Proof of Lemma \ref{lemma:ville_bound}}\label{sec:proof:lemma:ville_bound}

\begin{proof}
Fix $c>0$ and define the stopping time $\sigma:=\inf\set{n\ge 0:\ Z_n\ge 1/c}$. Note that since $\sigma(Z_t:t\leq n)\subseteq  \cF_n$, $\sigma$ is a $\cF_n$ stopping time. 

For $m\ge 0$, the bounded stopping time $\sigma\wedge m$ satisfies the optional sampling inequality for supermartingales:
\[
E{Z_{\sigma\wedge m}}\le E{Z_0}=1.
\]
On the event $\set{\sigma\le m}$, one has $Z_{\sigma\wedge m}=Z_\sigma\ge 1/c$, hence
\begin{align*}
1\ge E{Z_{\sigma\wedge m}} \stackrel{(i)}{\ge} E[\1\set{\sigma\leq m}Z_\sigma] \ge \frac{P\crbk{\sigma\le m}}{c}.
\end{align*}
where $(i)$ uses the nonnegativity of $Z$. 
Therefore $P\crbk{\sigma\le m}\le c$ for every $m$.
Letting $m\to\infty$ gives
\[
P\crbk{\sup_{n\ge 0} Z_n\ge \frac{1}{c}}=P\crbk{\sigma<\infty}\le c.
\]
\end{proof}

\subsection{Proof of Lemma~\ref{lemma:mixture-reduction}}\label{sec:proof:lemma:mixture-reduction}
\begin{proof}
Fix $\epsilon>0$ and recall from Lemma \ref{lemma:dirichlet-prior} that $\Pi(U(\epsilon,P))>0$.
By \eqref{eqn:Ueps-dominates},
\begin{align*}
\Lambda_n
&\ge\int_{U(\epsilon,P)} \prod_{t=0}^{n-1}\frac{Q(X_{t+1}| X_t)}{P_0(X_{t+1}| X_t)} d\Pi(Q)\\
&\ge\Pi(U(\epsilon,P)) e^{-\epsilon n} \prod_{t=0}^{n-1}\frac{P(X_{t+1}| X_t)}{P_0(X_{t+1}| X_t)}.
\end{align*}
    
Recall from \eqref{eqn:def_beta} that $\beta(\rho,\epsilon) = \log(1/(\rho\Pi(U(\epsilon,P))))$. Hence, on the event $\set{S_n(\epsilon)\ge \beta(\rho,\epsilon)}$ we have $\Lambda_n\ge 1/\rho$, so $\tau_\rho\le n$.
Since also $\tau_\rho\le n$ whenever $P_0(X_n| X_{n-1})=0$, taking the infimum over $n\ge 1$ yields $\tau_\rho\le \sigma(\epsilon)$.
\end{proof}

\subsection{Proof of Lemma~\ref{lemma:regen-crossing}}\label{sec:proof:lemma:regen-crossing}
\begin{proof}
Fix $y\in\R$ and write $R_k:=\sum_{i=1}^k M_i$ (with $R_0:=0$).
By definition,
\[
\set{N(y)>k}\subseteq   \set{y+R_k<\beta}.
\]
Choose $t>0$ such that $ E_{s_C}^Pe^{-tM_1}<\infty$. Since $E_{s_C}^PM_1>0$, the map $u\ra E_{s_C}^P e^{-uM_1}$ is differentiable at $u=0$ with derivative $-E_{s_C}^P M_1<0$, hence there exists $t_0$ such that for all $t\in(0,t_0]$ $E_{s_C}^Pe^{-tM_1}<1$.

For $t\in(0,t_0]$, let $\lambda(t):=-\log E_{s_C}^P e^{-tM_1}>0$.
Markov's inequality and independence give, for all $k\ge 1$,
\begin{align*}
P_{s_C}^P\crbk{N(y)>k}
&\le P_{s_C}^P\crbk{R_k<\beta-y}\\
&\le e^{t(\beta-y)}E_{s_C}^P e^{-tR_k}\\
&= e^{t(\beta-y)}\crbk{E_{s_C}^Pe^{-tM_1}}^k\\
&=\exp\sqbk{t(\beta-y)-k\lambda(t)}.
\end{align*}
Therefore,
\begin{align*}
E_{s_C}^PN(y)&=\sum_{k=0}^\infty P_{s_C}^P\crbk{N(y)>k}\\
&\le 1+\sum_{k=1}^\infty \min\set{1,\exp\sqbk{t(\beta-y)-k\lambda(t)}}\\
&\le 1+\frac{t}{\lambda(t)}(\beta-y)_+ + \frac{1}{1-e^{-\lambda(t)}}\\
&\le 2+\frac{t}{\lambda(t)}(\beta-y)_+ + \frac{1}{\lambda(t)}
\end{align*}
where the last inequality follows from the observation that $e^{x}\geq x+1$, hence $1-e^{-x}\geq \frac{x}{x+1}$.

Since $N(y)$ is a stopping time for the natural filtration of $\set{(\xi_k,M_k)}$ and $E_{s_C}^PN(y)<\infty$, Wald's identity yields
\[
g(y)=E_{s_C}^P\sqbk{\sum_{i=1}^{N(y)}\xi_i}=E_{s_C}^P\xi_1\cdot E_{s_C}^PN(y).
\]

By differentiability of $u\ra \log E_{s_C}^Pe^{-uM_1}$ at $0$ and $E_{s_C}^PM_1>0$, we may shrink $t$ (if needed) so that $\lambda(t)/t\ge \frac12 EM_1$. With this choice,
\begin{equation}\label{eqn:to_use_gy_first_bd}
g(y)\le \frac{2E_{s_C}^P\xi_1}{E_{s_C}^P M_1}(\beta-y)_+ + \underbrace{\crbk{2+ \frac{2}{tE_{s_C}^PM_1}}E_{s_C}^P\xi_1}_{:=\kappa_C}.
\end{equation}
Note that, under this choice of $t$, $\kappa_C$ only depends on $(P_0,P,C)$. 

Next, by the Markov renewal-reward identity for regenerative processes,
\begin{equation}\label{eqn:EM1}
\begin{aligned}
E_{s_C}^P M_1 
&= E_{s_C}^P\sqbk{\sum_{t=0}^{\xi_1-1}\crbk{\log\frac{P(X_{t+1}|X_t)}{P_0(X_{t+1}|X_t)}-\epsilon}}\\
&=E_{s_C}^P\xi_1 \sum_{s\in C}\nu_{P,C}(s)[\mrm{KL}\crbk{P(\cd | s)\|P_0(\cd | s)} - \epsilon]\\
&=\frac12 I_C(P\|P_0)\,E_{s_C}^P\xi_1
\end{aligned}
\end{equation}
where we recall that $\epsilon=\frac12 I_C(P\|P_0)$.

Finally, by \eqref{eqn:EM1},
\[
\frac{2E_{s_C}^P\xi_1}{E_{s_C}^P M_1}
=
\frac{2E_{s_C}^P\xi_1}{\frac12 I_C(P\|P_0)\,E_{s_C}^P\xi_1}
=
\frac{4}{I_C(P\|P_0)}.
\]
Plugging this into \eqref{eqn:to_use_gy_first_bd} proves \eqref{eqn:regen-crossing-bound}.
\end{proof}

\subsection{Proof of Lemma \ref{lemma:wc_unichain_pi}}\label{sec:proof:lemma:wc_unichain_pi}

\begin{proof}
Since $p$ is weakly communicating, let $C_p\subseteq S$ be as in Definition~\ref{def:wc}. It is standard for finite weakly communicating average-reward MDPs that there exists an average-reward optimal deterministic stationary policy \citep{puterman2014MDP}; fix one and denote it by $\bar\Delta$, and write $\bar P:=p_{\bar\Delta}$.

Let $\bar C$ be any closed communicating class of $\bar P$. Since every $s\in C_p^c$ is transient under every stationary policy, no recurrent class can intersect $C_p^c$, hence $\bar C\subseteq C_p$.

Define a directed graph on the vertex set $C_p$ by putting an edge $s\to s'$ if there exists an action $a\in A$ with $p(s'|s,a)>0$. The weak communication property implies that this graph is strongly connected; in particular, for every $s\in C_p$ there exists a directed path from $s$ to some state in $\bar C$.

For $s\in C_p$, let $d(s)$ denote the graph distance of $s$ from $\bar C$, i.e.,
\[
d(s):=\min\{n\ge 0:\exists\, s'\in \bar C\ \text{s.t. there is a directed path of length $n$ from $s$ to $s'$}\}.
\]
Then $d(s)=0$ if and only if $s\in \bar C$. Moreover, since $\bar C\subseteq C_p$ and the graph on $C_p$ is strongly connected, $d(s)<\infty$ for all $s\in C_p$.

For each $s\in C_p\setminus \bar C$, choose a state $\mathrm{nxt}(s)\in C_p$  and an action $a_s\in A$ such that
\[
d(\mathrm{nxt}(s))=d(s)-1,\quad\text{and}\quad  p(\mathrm{nxt}(s)\mid s,a_s)>0.
\]
Now define a deterministic stationary policy $\Delta$ by
\[
\Delta(\cd|s)=
\begin{cases}
\bar\Delta(\cd|s), & s\in \bar C,\\
\delta_{a_s}, & s\in C_p\setminus \bar C,\\
\bar\Delta(\cd|s), & s\in C_p^c.
\end{cases}
\]
Let $P:=p_\Delta$. 

Since $\Delta$ agrees with $\bar\Delta$ on $\bar C$, we have $P(\cdot\mid s)=\bar P(\cdot\mid s)$ for all $s\in \bar C$, and therefore $\bar C$ remains a closed communicating class of $P$.

Since all states in $C_p^c$ are transient under every stationary policy (in particular under $\Delta$), $P$ cannot have any closed communicating class contained in $C_p^c$.

Next, fix any $s\in C_p\setminus \bar C$. By construction, iterating $\mathrm{nxt}(\cdot)$ produces a directed path of length $d(s)$ that stays in $C_p\setminus \bar C$ until it reaches $\bar C$, and each transition along this path has strictly positive probability under $\Delta$. Hence
\[
P^{d(s)}(\bar C\mid s)>0,\qquad \forall s\in C_p.
\]
In particular, no closed communicating class of $P$ contained in $C_p$ can be disjoint from $\bar C$.

Combining the previous observations, any closed communicating class of $P$ must lie in $C_p$ and must intersect $\bar C$. Since $\bar C$ is itself closed, it follows that $P$ has $\bar C$ as its unique closed communicating class; i.e., $P$ is unichain. 

Finally, $\Delta$ coincides with the optimal policy $\bar\Delta$ on $\bar C$, and all states outside $\bar C$ are transient under $P$. Therefore, the long-run average reward under $\Delta$ equals the long-run average reward attained on $\bar C$ under $\bar\Delta$, which is $\alpha_p$. Hence, $\Delta$ is optimal.
\end{proof}

\subsection{Proof of Lemma \ref{lemma:stopped_TV_sum_bd}}
\begin{proof}
For $n\ge 0$, define
\[
M_n:=v^*(X_n)+\sum_{t=0}^{n-1}\crbk{r(X_t,A_t)-\alpha^*}.
\]

Since  $P = P_0$, by the Markov property under stationary policy $\Delta^*$,
\[
E_\mu^{\Delta^*,p}\sqbkcond{v^*(X_{n+1})}{\cH_n}
=
\sum_{s'\in S}P_0(s'|X_n)\,v^*(s').
\]
Then, by the definition of $v^*$ in \eqref{eqn:vstar_eval_eqn} 
\[
\sum_{a\in A}\Delta^*(a|X_n)r(X_n,a)-\alpha^*
=
v^*(X_n)-\sum_{s'\in S}P_0(s'|X_n)\,v^*(s')
=
v^*(X_n)-E_\mu^{\Delta^*,p}\sqbkcond{v^*(X_{n+1})}{\cH_n}.
\]
Therefore,
\begin{align*}
E_\mu^{\Delta^*,p}\sqbkcond{M_{n+1}-M_n}{\cH_n}
&=
E_\mu^{\Delta^*,p}\sqbkcond{v^*(X_{n+1})-v^*(X_n)+r(X_n,A_n)-\alpha^*}{\cH_n}\\
&=
E_\mu^{\Delta^*,p}\sqbkcond{v^*(X_{n+1})}{\cH_n}-v^*(X_n)+\sum_{a\in A}\Delta^*(a|X_n)r(X_n,a)-\alpha^*\\
&=0.
\end{align*}
Hence $\set{M_n:n\ge 0}$ is a $\cH_n$-martingale.

Since $S$ is finite, $v^*$ is bounded and $\theta\le T$ a.s., so $M_\theta$ is integrable and optional stopping applies:
\[
E_\mu^{\Delta^*,p}M_\theta=E_\mu^{\Delta^*,p}M_0=E_\mu^{\Delta^*,p}v^*(X_0).
\]
Rearranging this gives \eqref{eqn:stopped_sum_identity}.
Finally, $v^*(X_0)-v^*(X_\theta)\ge -\spnorm{v^*}$, which implies \eqref{eqn:stopped_sum_span_bd}.
\end{proof}

\subsection{Proof of Lemma \ref{lemma:RL_tv_lb}}\label{sec:proof:lemma:RL_tv_lb}
\begin{proof}
Recall the definition of online RL policies in Definition \ref{def:online_RL}. By the assumption that $p^*$ is unique, $\alpha_p > \alpha^*$. For all $\mu\in\cP(\cS)$
$$\begin{aligned}0 &= -\limsup_{T\ra\infty} \frac{1}{T}E_{\mu}^{\piRL,p}R_{r,p}(T)  \\
&=  \liminf_{T\ra\infty}\frac{1}{T}E_\mu^{\piRL,p}\sum_{t=0}^{T-1}[r(X_t,A_t) - \alpha^*+\alpha^* - \alpha_p],
\end{aligned}
$$
which means
$$
\alpha_p - \alpha^*= \liminf_{T\ra\infty}\frac{1}{T}E_\mu^{\piRL,p}\sum_{t=0}^{T-1}[r(X_t,A_t) - \alpha^*].
$$
In particular, this implies that there exists $T_{0,\mu}$ such that for all $T'\geq T_{0,\mu}$, 
$$ E_\mu^{\piRL,p}\sum_{t=0}^{T'-1}[r(X_t,A_t) - \alpha^*]\geq \frac{(\alpha_p-\alpha^*)T'}{2}.     
$$
But since $r(s,a) - \alpha^*\in[-1,1]$ and $0 < \alpha_p - \alpha^*\leq 2$, we have that for all $T\geq 0$
\begin{align*}
E_\mu^{\piRL,p}\sum_{t=0}^{T-1}[r(X_t,A_t) - \alpha^*]&\geq \frac{(\alpha_p-\alpha^*)T}{2}\1\set{T \geq T_{0,\mu}} -T\1\set{T<T_{0,\mu}}\\
&= \frac{(\alpha_p-\alpha^*)T}{2} -\crbk{\frac{(\alpha_p-\alpha^*)T}{2}+T}\1\set{T<T_{0,\mu}}\\
&\geq \frac{(\alpha_p-\alpha^*)T}{2} - 2T_{0,\mu}. 
\end{align*}
Therefore, letting $T_0 = \max_{s\in S}T_{0,\delta_s}$ completes the proof. 
\end{proof}

\end{document}